\documentclass[a4paper,twoside]{article}
\usepackage{fullpage}
\usepackage{epsfig}
\usepackage{amsmath}
\usepackage{amsthm}
\usepackage{amssymb}
\usepackage{pstricks, textpos}
\usepackage{graphicx}
\usepackage{mathrsfs}
\usepackage{multicol}
\usepackage{bbm}
\usepackage{fancyhdr}
\usepackage{enumitem}
\usepackage{esint}

\def\+{\oplus}

\newcommand{\R}{{\mathbb R}}

\newcommand{\N}{{\mathbb N}}

\newcommand{\cF}{{\mathcal F}}

\newcommand{\cE}{{\mathcal E}}

\newcommand{\cA}{{\mathcal A}}

\newcommand{\cC}{C}
\newcommand{\cH}{{\mathcal H}}

\newcommand{\dx}{\;\text{d}}
\newcommand{\ph}{\varphi}

\newcommand{\eps}{\varepsilon}
\newcommand{\one}{\mathbbm{1}}

\newcommand{\leqs}{\leqslant}
\newcommand{\geqs}{\geqslant}
\newcommand{\Omegai}{\Omega_\mathrm{int}}
\newcommand{\Omegae}{\Omega_\mathrm{ext}}
\newcommand{\tildeOmegai}{\widetilde \Omega_\text{int}}
\newcommand{\ui}{u_\mathrm{int}}
\newcommand{\ue}{u_\mathrm{ext}}
\newcommand{\vi}{v_\mathrm{int}}
\newcommand{\ve}{v_\mathrm{ext}}

\newcommand{\rhu}{\rightharpoonup}

\newcommand{\p}{p}

\newcommand{\limi}{\underline{\lim}~}
\newcommand{\lims}{\overline{\lim}~}

\newcommand{\ds}{\displaystyle}
\def\squareforqed{\hbox{\rlap{$\sqcap$}$\sqcup$}}
\def\qed{\ifmmode\else\unskip~ \fi\squareforqed}
\def\smartqed{\def\qed{\ifmmode\squareforqed\else{\unskip\nobreak\hfil
\penalty50\hskip1em\null\nobreak\hfil\squareforqed
\parfillskip=0pt\finalhyphendemerits=0\endgraf}\fi}}

\newcommand{\rat}{r}

\newtheoremstyle{thm}
{\topsep}
{\topsep}
{\itshape}
{}
{\sffamily \bfseries}
{.}
{ }
{\textsf{\thmname{#1}}\textsf{\textbf{\thmnumber{ \textup{#2}}}}\thmnote{~: \textit{#3}}}

\newtheoremstyle{def}
{\topsep}
{\topsep}
{\normalfont}
{}
{\sffamily \bfseries}
{.}
{ }
{\thmname{#1}\textbf{\textsf{\thmnumber{ \textup{#2}}}}\thmnote{~: \textit{#3}}}

\theoremstyle{thm}
\newtheorem{remark}{Remark}
\newtheorem{lemma}{Lemma}
\newtheorem{theorem}{Theorem}

\newtheorem{proposition}{Proposition}

\theoremstyle{def}
\newtheorem{definition}{Definition}
\newtheorem{assumption}{Assumption}

\title{\textsf{A transmission problem across a fractal self-similar interface}}
\author{
\textsf{Yves Achdou} \thanks {Universit{\'e} Paris Diderot, UFR Math{\'e}matiques, 
175 rue du Chevaleret, 75013 Paris, France, 
UMR 7598, Laboratoire Jacques-Louis Lions, F-75005, Paris, France, achdou@math.jussieu.fr},
\textsf{Thibaut Deheuvels}\thanks{{\'E}cole normale sup{\'e}rieure de Rennes, avenue Robert Schuman, 35170 Bruz, France, thibaut.deheuvels@ens-rennes.fr}
}
\date{}

\renewcommand{\subsectionmark}{\markboth{A transmission problem across a fractal self-similar interface}}

\makeatletter
\renewcommand{\section}{\@startsection {section}{1}{\z@}%
             {-3.5ex \@plus -1ex \@minus -.2ex}%
             {2.3ex \@plus.2ex}%
             {\normalfont\Large \sffamily \bfseries}}
\makeatother

\makeatletter
\renewcommand%
{\subsection}{\@startsection{subsection}{2}{0mm}
{-3.5ex \@plus -1ex \@minus -.2ex}%
             {2.3ex \@plus.2ex}%
{\normalfont\large \bfseries \sffamily}}%
\makeatother 

\makeatletter
\renewcommand%
{\subsubsection}{\@startsection{subsubsection}{3}{0mm}
 		{-3.5ex \@plus -1ex \@minus -.2ex}%
        	{2.3ex \@plus.2ex}%
{\normalfont \sffamily \bfseries }}%
\makeatother 

\makeatletter
\renewcommand%
{\paragraph}{\@startsection{paragraph}{5}{0mm}
{-0.8\baselineskip}{-1em}
{\normalfont \sffamily \bfseries }}%
\makeatother

\begin{document}

\begin{sffamily}
\maketitle
\end{sffamily}

\renewcommand{\abstractname}{\sf{Abstract}}

\thispagestyle{empty}

\begin{abstract}
We consider a transmission problem in which the interior domain has infinitely ramified structures.
 Transmission between the interior and exterior domains occurs only at the fractal component of the interface between  the interior and exterior domains.
We also consider the sequence of the transmission problems in which the interior domain is obtained by stopping the self-similar construction after a finite number of steps; 
the transmission condition is then posed on a prefractal approximation of the fractal interface. We prove the convergence in the sense of Mosco of the energy forms associated with these problems to the energy form of the limit problem. In particular, this implies the convergence of the solutions of the approximated problems to 
the solution of the problem with fractal interface. The proof relies in particular on an extension property.
\\
Emphasis is put on the geometry of the ramified domain. The convergence result is obtained when the fractal interface has no self-contact, and in a particular geometry with self-contacts, for which an extension result is proved.
\end{abstract}

\section{Introduction}\label{Introduction}

Transmission problems naturally arise in various fields of physics and have been extensively studied. An introduction to this class of problems can be found in \cite{MR0400916} and several applications are detailed in \cite{MR969367}.
\\
Such problems were more recently studied in the case when the interface is irregular, Lipschitz continuous or even fractal. These problems find many applications, such as the study of rough electrodes in electrochemistry or diffusion across irregular membranes in physiological processes, etc. (see \cite{PhysRevLett.73.3314,PhysRevLett.84.5776}). Several transmission problems with fractal interfaces have been studied in the case of the Koch flake or the Sierpi\'nski gasket in 2D and 3D (see \textit{e.g.} \cite{MR1916407, MR2056421, MR2323380, MR2679587, MR2645992, MR2557700, MR3100114, MR3100076}).

This paper deals with transmission problems between two domains $\Omegai$ and $\Omegae$ where $\Omegai$ is a ramified bounded domain as defined in Section \ref{domains} (see Figure \ref{fig-domram}). The domain $\Omegai$ presents infinitely many ramifications, and its boundary contains a fractal self-similar set $\Gamma$ which plays the role of the interface. 
\\
The domain $\Omegai$ can be seen as a bidimensional idealization of the bronchial tree, for example. Since the exchanges between the lungs and the circulatory system take place only in the last generations of the bronchial tree (the smallest structures), it is reasonable to consider transmission problems with specific transmission conditions accross the fractal boundary $\Gamma$ of the tree. We will however limit ourselves to simple transmission conditions.
\\
The fractal boundary $\Gamma$ belongs to a family of self-similar sets introduced by Mandelbrot \textit{et al.} in \cite{MR1684366}. Elliptic boundary value problems in the domain $\Omegai$ have been studied in \cite{ach-tch-2007-1}, and traces and extension results for these domains have been proved in \cite{ADT,SEP}.
\\[1mm]
The considered problem can be formally stated as
\begin{equation}\label{pb-transmission} \tag{$P$}
\left\{ \begin{array}{ll}
	-\Delta u = f & \mathrm{in}~ \Omegai\cup\Omegae,\\[1mm]
	[u]=0 &\text{on}~ \Gamma,\\[1mm]
	[\partial_n u]=\alpha u &\text{on}~ \Gamma,\\[1mm]
	\partial_n \ui =\partial_n \ue = 0 & \mathrm{on}~ \partial \Omegai \setminus \Gamma,\\[1mm]
	\partial_n \ue = 0,~~ \ui = u_0 & \mathrm{on}~ \Gamma^0,\\[1mm]
	\partial_n \ue =  0 & \mathrm{on}~ \partial D,
\end{array} \right.
\end{equation}
where $\alpha$ is a positive real number, $D$ is a regular bounded open domain in the plane with $\overline{D} = \overline{\Omegai}\cup \overline{\Omegae}$ such that $\Gamma \Subset D$, and $\Gamma^0$ is a line segment included in the boundary of $\Omegai$. The sets $\Omegai$ and $\Omegae$ are disjoint subdomains of $D$, and $[u]$ (\textit{resp.} $[\partial_n u]$) denotes the jump of $u$ (\textit{resp.} of the ``normal derivative" $\partial_n u$ of $u$) accross the fractal set $\Gamma$. 
Since the interface  $\Gamma$ is fractal, the normal derivative on $\Gamma$ has to be understood in a suitable weak sense, which will be made precise later.\\
Problem \eqref{pb-transmission} is a model problem. Its study is the first step in the modelling of physical transmission problems in ramified structures.
\\[1mm]
The goal is to study approximations $(P_n)$ of problem \eqref{pb-transmission}, obtained by stopping the construction of the ramified domain $\Omegai$ at step $n$. The interfaces in the problems $(P_n)$ are called prefractal approximations of the fractal set $\Gamma$. They consist of disjoint finite unions of line segments, which makes problems $(P_n)$ much simpler than problem $(P)$. A natural question is to understand the asymptotic behavior of the problems $(P_n)$ as $n\to \infty$, and in particular to investigate the convergence of the solutions $u_n$ of the problems $(P_n)$ to the solution $u$ of problem $(P)$.
\begin{remark}\label{sec:introduction-1}
In contrast with the references \cite{MR1916407, MR2056421, MR2323380, MR2679587, MR2645992, MR2557700, MR3100114, MR3100076}, the boundary value problem does not only involve transmission conditions at the interface between $\Omegai$ and $\Omegae$:  there are also homogeneous Neumann conditions on the polygonal part of $\partial \Omegai$; as a consequence, the traces of
 $\ui$  and $\ue$ on this set do not match a priori.  Coping with these discontinuities will be a difficulty in studying the convergence of the solutions $u_n$ of the problems $(P_n)$ to the solution $u$ of problem $(P)$.\end{remark}
\begin{remark}
\label{sec:introduction}
We have chosen that the source term in ($P$)  appear both in the  Dirichlet boundary condition on $\Gamma_0$ for $\ui$ and in the Poisson equations in $\Omegai$ and $\Omegae$; 
this is of course completely arbitrary.
\end{remark}
A crucial step in the study of the asymptotic behavior of $(P_n)$ is the question of extending functions defined in the domain $\Omegai$. More precisely, it is of particular importance that $\Omegai$ should satisfy a $W^{1,p}$-extension property for some $p\in [1,\infty]$, \textit{i.e.} there should exist a bounded linear operator 
	$$\cE : W^{1,p}(\Omegai)\to W^{1,p}(\R^2)$$
such that ${\cE (u)}_{|\Omegai} = u$ for all $u\in W^{1,p}(\Omegai)$.
The  domains satisfying this property for all $p\in [1,\infty]$ will often be referred to  as {\it Sobolev extension domains}.
\\
It is well known that every Lipschitz domain in $\R^n$, that is every domain whose boundary is locally the graph of a Lipschitz function, is a Sobolev extension domain. Calder{\'o}n proved the $W^{1,\p}$-extension property for $\p \in (1,\infty)$ (see \cite{MR0143037}), and Stein extended the result to the cases $\p=1$ and $\p=\infty$ (see \cite{MR0290095}).\\
In \cite{MR631089}, Jones generalized this result to the class of $(\eps,\delta)$-domains, also referred to as Jones domain, or locally uniform domains (see \cite{MR565886}). In dimension two, the definition of $(\eps,\delta)$-domains is equivalent to that of quasi-disks, see \cite{MR817985}.
This extension result is almost optimal in the plane, in the sense that every plane finitely connected Sobolev extension domain is an $(\eps,\delta)$-domain, see \cite{MR631089,MR817985}. The case of an unbounded domain in $\R^2$ has been studied in \cite{MR530508}. Extension properties for domains which do not have the $(\eps,\delta)$-property have been studied \textit{e.g.} in \cite{MR1643072}, where the authors examine in particular the case of domains with cusps. 
\\
\indent
When the fractal boundary $\Gamma$ of the domains $\Omegai$ studied in this paper has no self-contact, $\Omegai$ can be proved to be an $(\eps,\delta)$-domain, and and is therefore a Sobolev extension domain. However, in the case when the boundary self-intersects, $\Omegai$ does not have the $W^{1,p}$-extension property for all $p\geqs 1$. In particular, the extension property does not hold for $p=2$, which is the relevant case here,  since the variational formulations  of (\ref{pb-transmission}) naturally involve the spaces  $H^1(\Omegai)$.
 In the particular geometry considered in Section \ref{r=rstar}, it will be proved that the transmission condition imposed on $\Gamma$ 
yields a better regularity of the trace on $\Gamma$, for functions belonging to the  function space arising in the variational formulation. It is then possible to deduce an extension result in this case (see Theorem \ref{ext-thm}).
\\

The main question investigated in this paper is the question of the convergence in the sense of Mosco of the energy forms associated with problem $(P_n)$ to the energy form of the problem $(P)$. 
The notion of Mosco-convergence, or $M$-convergence, was introduced in \cite{MR0298508}, see also \cite{MR1283033}. It is a stronger convergence in the space of forms than $\Gamma$-convergence. In particular, it also implies the convergence of minimizers of the energy forms to the minimizer of the limit form. The $M$-convergence of forms is equivalent to the convergence of the resolvent operators associated with the relaxed forms in the strong operator topology (see \cite{MR1283033}). 

The main results of this paper are Theorems \ref{M-cv} and \ref{ext-thm}. Theorem \ref{M-cv}  is about the convergence of the energy forms associated with $(P_n)$ in the sense of Mosco to the energy form associated with \eqref{pb-transmission}, in the case when the fractal interface has no self-contact. The proof uses the extension operator from $H^1(\Omegai)$ to $H^1(\R^2)$ as a main ingredient. 
The existence of a continuous extension operator from $H^1(\Omegai)$ to $H^1(\R^2)$, for a particular geometry where the fractal interface self-intersects,  is stated in Theorem \ref{ext-thm}. As a consequence, the proof of Theorem \ref{M-cv} can be reproduced in this case, to show the $M$-convergence of the energy forms.

The paper is organized as follows: the geometry of the interior and exterior domains is detailed in Section \ref{sec-geometry}, as well as the prefractal geometry. Section \ref{sec-function-spaces} is devoted to the study of the function spaces involved in the paper, and emphasis is put on trace and extension results for the domains under study. 
The considered transmission problem is described in Section \ref{sec-transmission-pb}. 
Section \ref{sec-convergence} is devoted to the $M$-convergence of the energy forms associated with the problem with prefractal interface to the energy form of the problem 
with fractal interface, in the case when the boundary of $\Omegai$ has no self-contact. 
In Section \ref{r=rstar}, a particular geometry in which the fractal part of the boundary of $\Omegai$ self-intersects is considered;
an extension result is proved in this particular case and the $M$-convergence of the energy forms follows.

%
\section{The geometry}\label{sec-geometry}
\subsection{The fractal interface}\label{fractal-interface}
\subsubsection{Definitions}\label{def}
Consider four real numbers $\rat, \beta_1,\beta_2, \theta$ such that $1/2\leqs \rat< 1/\sqrt{2}$, $\beta_1>0$, $\beta_2>0$ and  $0 \leqs \theta<\pi/2$.
Let  $f_i$, $i=1,2$ be the  two similitudes in $\R^2$ given by
\begin{equation*}
    f_1  \begin{pmatrix} x_1\\x_2 \end{pmatrix}
      ~=~
      \begin{pmatrix} -\beta_1 \\ \beta_2 \end{pmatrix}
	+\rat\begin{pmatrix} x_1 \cos \theta- x_2 \sin \theta\\x_1 \sin \theta+ x_2 \cos \theta \end{pmatrix}, \end{equation*}
\begin{equation*}
    f_2 \begin{pmatrix}  x_1\\ x_2 \end{pmatrix}
    ~=~ \begin{pmatrix}  \beta_1\\ \beta_2 \end{pmatrix}
    +\rat \begin{pmatrix}  x_1 \cos \theta + x_2 \sin \theta\\-x_1 \sin \theta+ x_2 \cos \theta \end{pmatrix}.
\end{equation*}
The two similitudes have the  same dilation ratio $\rat $ and  opposite angles $\pm \theta$. One can obtain $f_2$ by composing $f_1$ with the symmetry with respect to the vertical axis $\{x_1=0\}$.\\\noindent
Let $\Gamma$ denote the self-similar set associated with the similitudes $f_1$ and $f_2$, \textit{i.e.} $\Gamma$ is the unique compact subset of $\R^2$ such that
\begin{equation*}
	\Gamma = f_1(\Gamma)\cup f_2(\Gamma).
\end{equation*}
It was stated in \cite{MR1684366} (see \cite{deheuvelsphd} for a complete proof) that for any $\theta$, $0\leqs \theta<\pi/2$, there exists a unique positive number $\rat ^\star_\theta\in [1/2, 1 /{\sqrt 2}[$ which only depends on the angle $\theta$ such that 
\begin{equation*}
\label{SEP:eq:10}
\begin{array}[c]{ll}
\diamond~~ \hbox{if}~ 0<\rat<\rat^\star_\theta, &\hbox{then}~ \Gamma ~\hbox{is totally disconnected,} \\
\diamond~~\hbox{if}~ \rat=\rat^\star_\theta, &\hbox{then}~  \Gamma \hbox{ is connected}.         \end{array}
\end{equation*}
In the following, we will always assume that $\rat \leqs \rat^\star_\theta$.
\paragraph{Notations}
For every integer $n>0$, we note $\cA_n=\{1,2\}^n$. For $\sigma\in \cA_n$, we note $f_\sigma$ the similitude $f_{\sigma_1}\circ\ldots \circ f_{\sigma_n}$. We agree to extend the notation to the case $n=0$: $f_\sigma=\text{Id}$ if $\sigma\in \cA_0$. We also introduce the notation $\cA := \bigcup_{n\geqs 0} \cA_n$.
\\[2mm]
For $\sigma\in \cA$, we note $\Gamma^\sigma = f_\sigma(\Gamma^0)$, and for every integer $n\geqs 0$, $\displaystyle \Gamma^n = \bigcup_{\sigma\in {\cA_n}} f_\sigma(\Gamma^0)$.

\subsubsection{Hausdorff dimension of $\Gamma$}\label{dimension-gamma}
If $\rat\leqs \rat^\star_\theta$, then it can be seen that the open set condition (or Moran condition) holds, see \cite{MR0014397} or \cite{MR1840042} for a definition. The open set condition is satisfied \textit{e.g.} for the domain $\Omegai$ defined in \eqref{def-dom-ram}, if Assumption \ref{assumption1} below is satisfied (Theorem \ref{exist-ram-dom} proves the existence of such a domain).
\\ 
The open set condition implies that the Hausdorff dimension of $\Gamma$ is
	$$d:= \dim_H \Gamma = -\frac{\log 2}{\log \rat}$$
see \cite{MR0014397, MR1840042}. If $0\leqs \theta<\pi/2$, then  $1/2 \leqs \rat \leqs \rat^\star_\theta<1/\sqrt{2}$ and thus $1\leqs d <2$.
\\[1mm]
In the case when $\rat=\rat^\star_\theta$, introduce the set
\begin{equation}\label{Xi}
\Xi = f_1(\Gamma)\cap f_2(\Gamma).
\end{equation}
In this case, the fractal set $\Gamma$ self-intersects, and 
union of the images of $\Xi$ by the similitudes $f_{\sigma_1}\circ \ldots \circ f_{\sigma_n}$, $\sigma_1,\ldots,\sigma_n \in \{1,2\}$ is the set of the multiple points of $\Gamma$.
\\
Two situations can occur, depending on the angle $\theta$ (see \cite{MR1684366}):
\begin{list}{--}{\itemsep1mm \topsep1mm}
\item if $\theta$ is not of the form $\frac\pi{2k}$ for any integer $k>0$, then $\Xi$ is reduced to a single point, and $\Gamma$ has countably many multiple points,
\item if $\theta$ is of the form $\frac\pi{2k}$ where $k>0$ is an integer, then $\Xi$ is a Cantor set, and the Hausdorff dimension of $\Xi$, noted $\dim_H \Xi$, is $\frac{\dim_H \Gamma}2$.
\end{list}

\subsubsection{The self-similar measure $\mu$}
Recall the classical result on self-similar measures,  see  \cite{MR1449135,MR625600} and \cite{MR1840042} page 26.
\begin{theorem}
There exists a unique Borel regular probability measure $\mu$ on $\Gamma$ such that for any 
Borel set $A\subset\Gamma$,
\begin{equation}
  \label{SEP:eq:18} 
\mu(A)=\frac 1 2 \mu\left(f_{1}^{-1}(A)\right)+\frac 1 2 \mu \left(f_{2}^{-1}(A)\right).
\end{equation}
\end{theorem}
\noindent The  measure $\mu$ is called the {\it self-similar measure defined in the self-similar 
triplet $ \left(\Gamma, f_1,f_2\right)$}. 
\\[1mm]
Let $L^\p_\mu(\Gamma)$, $\p \in [1,+\infty)$ be the space of the measurable functions $v$ on $\Gamma$ such that 
$\int_{\Gamma}|v|^\p \dx \mu<\infty$, endowed with the norm  $\Vert v\Vert_{L^\p(\Gamma)}= \left(\int_{\Gamma}|v|^\p \dx \mu\right)^{1/p}$. 
A Hilbert basis of $L^2_\mu(\Gamma)$ can be constructed \textit{e.g.} with Haar wavelets.\\[1mm]
The space $W^{s,\p}(\Gamma)$ for $s\in (0,1)$ and $\p \in [1,\infty)$ is defined as the space of functions $v\in L^\p(\Gamma)$ such that ${\vert v \vert}_{W^{s,\p}(\Gamma)} < \infty$, where
	\[ {\vert v \vert}_{W^{s,\p}(\Gamma)} ~=~ {\left( \int_{\Gamma} \int_{\Gamma} \frac{{\vert v(x)-v(y) \vert}^\p}{{\vert x-y\vert}^{d+\p s}} \dx \mu(x) \dx \mu(y) \right)}^\frac 1 \p. \]
Endowed with the norm ${\Vert v \Vert}_{W^{s,\p}(\Gamma)} = {\Vert v \Vert}_{L^\p(\Gamma)} +  {\vert v \vert}_{W^{s,\p}(\Gamma)} $, the spaces $W^{s,p}(\Gamma)$ are Banach spaces. In the special case $p=2$, the space $W^{s,p}(\Gamma)$ is a Hilbert space, and is noted $H^s(\Gamma)$.

\begin{remark}\label{norme-Lip}
In the special case when $\theta=0$ and $r=r^\star_\theta=1/2$, the set $\Gamma$ is in fact a line segment. This geometry will be discussed in Section \ref{r=rstar} (see Figure \ref{trapeze}).
In this case, it can be proved that if $s\in (0,1)$, then an equivalent norm of the space $W^{s,p}(\Gamma)$ is given by
\begin{equation}\label{H-Lip}
{\Vert v \Vert}_{\mathrm{Lip}_s^{p,p}(\Gamma)}^p ~:= ~ \int_\Gamma {|v|}^p \dx\mu + \sum_{k\geqs 0} 2^{skp} \sum_{\sigma\in \cA_k} \int_{2f_\sigma  (\Gamma)} {|v-{\langle v \rangle}_{2f_\sigma(\Gamma)}|}^p\dx x
\end{equation}
where $2f_\sigma(\Gamma)$ is the intersection with $\Gamma$ of the segment obtained by expanding the line segment $f_\sigma(\Gamma)$ with a factor 2 around its center (see \cite{MR2032227,MR820626}). As in the rest of the paper, if $v$ is a measurable function in a measured space $(X,m)$, the notation $\langle v \rangle_X$ refers to the mean value $\frac 1{m(X)} \int_X v \dx m$.\end{remark}

\subsection{The domains $\Omegai$ and $\Omegae$}\label{domains}
\noindent Call $P_1=(-1,0)$ and $P_2=(1,0)$ and $\Gamma^0$ the line segment $\Gamma^0=[P_1 P_2]$.  Let us assume  that $f_2(P_1)$ and $f_2(P_2)$ have positive coordinates, \emph{i.e.} that
\begin{equation}
  \label{cond1}
\rat \cos\theta <\beta_1
~ \text{ and } ~ \rat \sin\theta < \beta_2.
\end{equation} 
\noindent Let us also assume that the open domain $Y^0$ inside the closed polygonal line joining the points $ P_1$, $P_2$, $f_2(P_2)$, $f_2(P_1)$, $f_1(P_2)$, $f_1(P_1)$, $P_1$ in this order must be convex and hexagonal, except if $\theta=0$, in which case it is trapezoidal. With (\ref{cond1}), this is true if and only if
\begin{equation}
  \label{cond2}
(\beta_1-1)\sin\theta +\beta_2 \cos\theta > 0.
\end{equation}

\noindent Under assumptions (\ref{cond1}) and (\ref{cond2}), the domain  $Y^0$ is contained in the half-plane $x_2>0$ and symmetric with respect to the vertical axis $x_1=0$.
\\[1mm]
Call $K^0= \overline{ Y^0}$. It  is possible to glue together $K^0$, $f_1(K^0)$ and $f_2(K^0)$ and obtain a new polygonal domain. The  assumptions \eqref{cond1} and \eqref{cond2} imply that $Y^0\cap f_1( Y^0)=\emptyset$ and $ Y^0\cap f_2( Y^0)=\emptyset$.
\begin{figure}[h!]
\centering
\scalebox{0.7}{\input{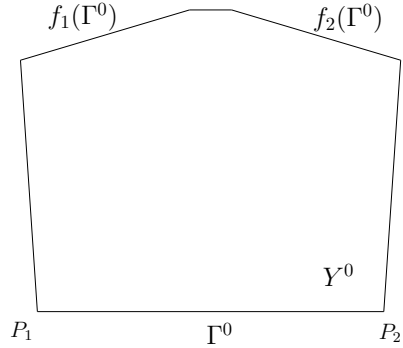}}
\caption{The construction of the first cell $Y^0$}
\end{figure}

\noindent 
Let the open domain $\Omegai$ (see Figure \ref{fig-domram}) be defined as follows:
\begin{equation}\label{def-dom-ram}
 \Omegai = \hbox{Interior} \left(\bigcup_{\sigma\in \cA} f_\sigma (K^0)\right),
\end{equation}
with the notations of \S \ref{def}.
\\[3mm]
For a given $\theta$, with $\rat^\star_\theta$ defined as above, the following assumption on $(\alpha,\beta)$ will be made:
\begin{assumption}\label{assumption1}
For $0\leqs \theta<\pi/2$, the parameters $\alpha$ and $\beta$ satisfy \eqref{cond1} and \eqref{cond2} for $\rat=\rat^\star_\theta$, and are such that
\begin{equation*}
\left\{
\begin{array}[c]{l}
\hbox{\textit{i}. for all $\rat $, $0<\rat \leqs \rat^\star_\theta$, the sets $Y^0$, $f_\sigma ({ Y^0})$, $\sigma\in \cA$, are pairwise disjoint,}\\
\hbox{\textit{ii}.  for all $\rat $, $0<\rat< \rat^\star_\theta$, $f_1(\overline {\Omega}_\text{int})\cap f_2(\overline {\Omega}_\text{int})=\emptyset$,}\\
\hbox{\textit{iii}. for $\rat =\rat^\star_\theta$, $f_1(\overline{\Omega}_\text{int}) \cap f_2(\overline{\Omega}_\text{int})\not=\emptyset$.} 
\end{array}
\right.
\end{equation*}
\end{assumption}

\begin{remark}
Assumption \ref{assumption1} implies that if $r=r^\star_\theta$, then $f_1( {\Omegai})\cap f_2( {\Omegai})=\emptyset$.
\end{remark}
%
\noindent
In the case $\theta=0$, Assumption \ref{assumption1} is satisfied by any $\alpha>\rat^\star_\theta=1/2$ and $\beta>0$. The following theorem, proved in \cite{YANT2010}, asserts that for all $\theta\in (0,\pi/2)$, there exists $(\alpha,\beta)$ satisfying Assumption 1. 

\begin{theorem}\label{exist-ram-dom}
If $\theta\in (0, \pi/2)$,  then for every $\alpha>\rat^\star_\theta \cos \theta$, there exists $\bar \beta>0$ such that for all $\beta\geqs \bar \beta$, $(\alpha,\beta)$ satisfies Assumption \ref{assumption1}.
\end{theorem}

Let $D$ be an open bounded domain with a Lipschitz boundary, containing the closure of $\Omegai$. The exterior domain $\Omegae$ is defined by 
\begin{equation}
\Omegae := \text{Interior} (D\setminus \Omegai).
\end{equation}

\begin{remark}
The assumption that $ \Omegai \Subset D$ may be relaxed: in fact, it would be enough to assume that $\Gamma \Subset D$.
\end{remark}

Displayed on Figure \ref{fig-domram} are examples of the domains $\Omegai$ and $\Omegae$,  for the parameters $\theta=\pi/5$ in the left-hand side and $\theta=\pi/4$ in the right-hand side. 

\begin{figure}[h!]
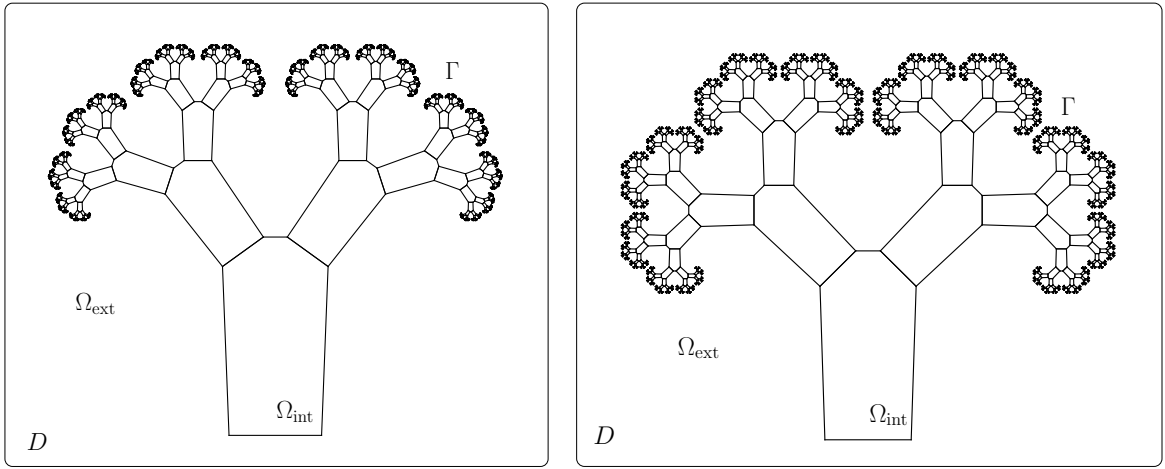

   \centering\[
\scalebox{0.45}{\input{fig1.pstex_t}}
~~~
\scalebox{0.4663}{\input{fig2.pstex_t}}
  \]
   \caption{The  ramified domain $\Omega$ for $\theta=\pi/5$ (left) and $\theta=\pi/4$ (right) when $\rat <\rat^\star_\theta$, $\beta_1=0.7$, $\beta_2=4$.}
   \label{fig-domram}
 \end{figure}

\subsection{The truncated domain $\Omegai^n$ and the prefractal interface}
For every integer $n\geqs 0$, the truncated domain $\Omegai^n$ is defined by
\begin{equation}
 	\Omegai^n = \text{Interior}\left(\bigcup_{0\leqs k\leqs n} \bigcup_{\sigma \in {\cA_k}} f_\sigma(K^0)\right),
\end{equation}
with the notations of \S \ref{def}. As above, the exterior domain associated to $\Omegai^n$  is 
\begin{equation}
	 \Omegae^n = \text{Interior}(D\setminus \Omegai^n).
\end{equation}
Note that the set $\Gamma^n$ defined in \S \ref{def} is a part of the boundary of $\Omegai^n$. The sets $\Gamma^n$, $n\geqs 0$, will be referred to as \textit{prefractal approximations} of the fractal set $\Gamma$.

\section{Function spaces}\label{sec-function-spaces}

Hereafter, we consider a domain $\Omegai$ as defined in \ref{domains}, with $\theta$ in $[0,\pi/2)$ and $\rat \leqs \rat^\star_\theta$, and we assume that the parameters $\alpha$, $\beta$ are such that Assumption \ref{assumption1} is satisfied.
\\[1mm]
We define $W^{1,p}(\Omega)$, $p=[1,\infty]$, $\Omega = \Omegai$ or $\Omega=\Omegae$, to be the space of functions in $L^p(\Omega)$ with first order partial derivatives in $L^p(\Omega)$.\\ 
The space $W^{1,p}(\Omega)$ is a Banach space with the norm
$
  \left(\|u\|_{L^p(\Omega)}^p +\| \frac {\partial u}{\partial x_1  }\|_{L^p(\Omega)}^p  +\| \frac {\partial u}{\partial x_2  }\|_{L^p(\Omega)}^p \right)^{1/p}$,
see for example \cite{MR2424078}. Elementary calculus shows that
$
\|u\|_{W^{1,p}(\Omega)} := \left(\|u\|_{L^p(\Omega)}^p +\| \nabla u\|_{L^p(\Omega)}^p
\right)^{1/p}
$
is an equivalent norm, with
$\| \nabla u\|^p_{L^p(\Omega)}:=   \int_{\Omega} |\nabla u|^p $ and 
$|\nabla u|=\sqrt{|\frac {\partial u}{\partial x_1}|^2 + |\frac {\partial u}{\partial x_2}|^2}$.
\\[1mm]
In the special case $p=2$, the space $W^{1,p}(\Omega)$ is a Hilbert space, and is noted $H^1(\Omega)$.\\[1mm]
The spaces $W^{1,p}(\Omegai)$ as well as elliptic boundary value problems in $\Omegai$ have been studied in  \cite{ach-tch-2007-1}, with, in particular Poincar{\'e} inequalities and a Rellich compactness theorem. The same results in a similar but different geometry were proved by Berger \cite{MR1800198} with other methods.

\subsection{Trace results}

\subsubsection{The classical definition of traces}
We recall the classical definition of a trace operator on $\partial \omega$ when $\omega$ is an open subset of $\R^2$ (see for instance \cite{MR820626} p. 206).
\begin{definition}\label{def-trace}
Consider an open set $\omega\subset \R^2$. The function  $u\in L^1_{loc}(\omega)$ can be strictly defined at $x\in\overline \omega$ if the limit
\begin{equation}
  \label{eq:25}
{\overline{u}}(x)=\lim_{r\rightarrow 0}\frac{1}{|B(x,r)\cap \omega|}\int_{B(x,r)\cap \omega}u(y) \dx y
\end{equation}
exists, where $|B(x,r)\cap \omega|$ is the 2-dimensional Lebesgue measure of the set $B(x,r)\cap \omega$. In this case, $x$ is said to be a Lebesgue point of $u$.
\\
The trace $u_{|\partial \omega}$ is defined to be the function given by $u_{|\partial \omega}(x)= {\overline{u}}(x)$ at every point $x\in \partial \omega$ such that the limit ${\overline{u}}(x)$ exists.
\end{definition}

\begin{remark}
Recall that for any $p>1$, a function which belongs to $W^{1,p}(\R^n)$ can be strictly defined except on a set with zero $p$-capacity, see for example \cite{MR0435361} and \cite{MR1404091}.
\end{remark}

\subsubsection{A trace theorem on $\Gamma$}
\noindent
It has been shown in \cite{comparison} (see Theorem 11) that every function in $W^{1,p}(\Omegai)$ can be strictly defined on $\Gamma$ $\cH^1$-almost everywhere, where $\cH^1$ is the one-dimensional Hausdorff measure. Moreover, the following trace result holds.

\begin{theorem} (see \cite{ADT})\label{trace-thm}
\begin{list}{\textbullet}{\topsep1mm \itemsep0mm}
\item Assume $\rat<\rat^\star_\theta$. For all $p\in ]1,\infty]$, if $u\in W^{1,p}(\Omegai)$, then $u_{|\Gamma}\in W^{1-\frac{2-d}p,p}(\Gamma)$, and there exists a constant $C>0$ independent of $u$ such that
	$${\|u_{|\Gamma}\|}_{W^{1-\frac{2-d}p,p}(\Gamma)} \leqs C {\|u\|}_{W^{1,p}(\Omegai)}.$$
\item Assume $\rat=\rat^\star_\theta$, then
\begin{list}{$\diamond$}{\topsep0mm \itemsep0mm}
\item the previous result holds if $1<p<2-\dim_H \Xi$,
\item if $p\geqs 2-\dim_H \Xi$, then 
	$W^{1,p}(\Omegai)_{|\Gamma} \subset W^{s,p}(\Gamma) ~ \text{for all } s<\frac 1 p (d-\dim_H \Xi),$
and the embedding is continuous. Moreover, if $s>\frac 1 p (d-\dim_H \Xi)$, then ${W^{1,p}(\Omegai)}_{|\Gamma} \not\subset W^{s,p}(\Gamma)$.
\end{list}
\end{list}
\end{theorem}

\begin{remark}
  \label{sec:trace-theorem-gamma}
The space of the traces on $\Gamma$ of  functions in $W^{1,p}(\Omegai)$,  $1<p<\infty$ 
was characterized in \cite{YANT2010}, whether $r< r^\star_\theta$ or $r=r^\star_\theta$, as the space  
 ${ \rm{JLip}}(1-\frac{2-d} p,p,p;0; \Gamma)$, which was first introduced in \cite{MR2032227}.
Of course, if  $r< r^\star_\theta$, then   ${\rm{JLip}}(1-\frac{2-d} p,p,p;0; \Gamma)$ coincides with $W^{1-\frac{2-d} p,p}(\Gamma)$. 
An easy consequence of this characterization is that the space of the traces on $\Gamma$ of functions in $W^{1,p}(\Omegai)$ is relatively compact in $L^p_\mu (\Gamma)$.
\end{remark}


\noindent
Proposition \ref{convergence-result} below will be useful in the proofs of the main theorems of this paper.

\begin{proposition}\label{convergence-result}
For every $u\in H^1(\Omegai)$, 
	$$\frac 1 {|\Gamma^n|} \int_{\Gamma^n} {u_{|\Gamma^n}}^2\dx x \underset{n\to \infty}{\longrightarrow} 
		\int_\Gamma {u_{|\Gamma}}^2\dx \mu.$$
\end{proposition}

\begin{proof}
The present proof relies on  Proposition 1 in \cite{ach-tch-2007-1}, which states that
 for any $u\in H^1(\Omegai)$, the sequence of piecewise constant functions $(\tilde u_n)_{n\in \N}$ defined on $\Gamma$: 
\begin{displaymath}
\tilde u_n   =\sum_{\sigma\in \cA_n}   \left\langle u \right\rangle_{\Gamma^\sigma}    \one_{f_\sigma( \Gamma) } ,
\end{displaymath}
where $  \left\langle u \right\rangle_{\Gamma^\sigma}= \frac {1}  {|\Gamma^\sigma|} \int_{\Gamma^\sigma} u_{|\Gamma^\sigma}(z) \dx z$,
is such that $\lim_{n\to \infty} \|\tilde u_n -u_{|\Gamma} \|_{L_\mu ^2(\Gamma)}=0$. Note also that 
$
  \int_{\Gamma}  \tilde u_n ^2 d\mu= \sum_{\sigma\in \cA_n}   \frac {| \Gamma^\sigma|}{|\Gamma^n|} 
 \left\langle u \right\rangle_{\Gamma^\sigma} ^2 $.
 Hence, in order to prove Proposition \ref{convergence-result}, it is enough to  prove that\\
$  \left|   \frac 1 {|\Gamma^n|} \int_{\Gamma^n} {u_{|\Gamma^n}}^2\dx x -  \int_{\Gamma}  \tilde u_n ^2 d\mu\right| \underset{n\to \infty}{\longrightarrow}0$, or in an equivalent manner, that
  $  S_n \underset{n\to \infty}{\longrightarrow}0$, where 
\begin{displaymath}
  S_n=     \frac 1 {|\Gamma^n|} \sum_{\sigma\in \cA_n} \left( \int_{\Gamma^\sigma }   u^2_{|\Gamma^\sigma}(z) \dx z -   | \Gamma^\sigma|
 \left\langle u \right\rangle_{\Gamma^\sigma} ^2 \right)  =      \frac 1 {|\Gamma^n|} \sum_{\sigma\in \cA_n} \int_{\Gamma^\sigma} \left[ u_{|\Gamma^\sigma}(z)- \left\langle u \right\rangle_{\Gamma^\sigma}  \right ]^2 \dx z.    
\end{displaymath}
 From a standard trace result on  $\Gamma^0$ and appropriate rescalings, we know that for a positive constant independent of $n$, $\sigma\in \cA_n$ and $u\in H^1 (\Omegai)$,  
   $\int_{\Gamma^\sigma} \left[ u_{|\Gamma^\sigma}(z)- \left\langle u \right\rangle_{\Gamma^\sigma}  \right ]^2 \dx z \le C |\Gamma^\sigma| \int_{f_\sigma( \Omegai)} |\nabla u|^2$.
 Hence, $ S_n\le  C  \sum_{\sigma\in \cA_n}  \frac {|\Gamma^\sigma|}{|\Gamma^n|}  \int_{f_\sigma( \Omegai)} |\nabla u|^2= 
 \frac C{2^n} \sum_{\sigma\in \cA_n} \int_{f_\sigma( \Omegai)}   |\nabla u|^2$, which implies that $S_n \underset{n\to \infty}{\longrightarrow}0$.
\end{proof}

\noindent
We also recall the following refined trace inequality, we refer to \cite{YANT2010} for the proof.
\begin{theorem}\label{SPI}
[see \cite{YANT2010}, Th. 11]
Assume that $\rat\geqs 1/2$, then for all real number $\kappa \in (2\rat^2,1)$, there exists a constant $C$ such that for all $v\in H^1(\Omegai)$ with $v_{|\Gamma^0}=0$,
\begin{equation}
{\Vert v_{|\Gamma}\Vert}_{L^2_\mu(\Gamma)}^2 ~\leqs ~ C \sum_{m\geqs 0} \kappa^m \sum_{\tau \in \cA_n} {\Vert \nabla v\Vert}^2_{L^2(f_\tau(Y^0))}.
\end{equation}
\end{theorem}

\subsection{Extension results} 

\subsubsection{The subcritical case $\rat <\rat^\star_\theta$}\label{ext-subcritical}
As seen in \S~\ref{Introduction}, it was proved in \cite{YANT2008} that if $\rat <\rat^\star_\theta$, then $\Omegai$ is an $(\eps,\delta)$-domain (see \cite{MR631089} for a definition), or in an equivalent manner, a quasi-disk (see \cite{MR817985}). Hence, the extension result of Jones  and Vodop'janov {\textit{et al.}} applies, and $\Omegai$ is a Sobolev extension domain (see \cite{MR631089}), \textit{i.e.} $\Omegai$ has the $W^{1,p}$-extension property for every $p\in [1,\infty]$: there exists a continuous linear operator $\cE$ from $W^{1,p}(\Omegai)$ to $W^{1,p}(\R^2)$ such that 
    \begin{equation}\label{ext-op}
      \cE(u)_{|\Omegai}=u, \quad \forall u\in W^{1,p}(\Omegai).
    \end{equation}
Similarly, the set $\Omegae = D \setminus \overline \Omega_\text{int}$ is an $(\eps,\delta)$-domain, and thus a Sobolev extension domain.

\subsubsection{The critical case $\rat =\rat^\star_\theta$}
When $\rat =\rat^\star_\theta$, it is easily seen that $\Omegai$ is not an $(\eps,\delta)$-domain, and the extension results of Jones and Vodop'janov {\textit{et al.}} do not apply. In fact, if $p\in (1,\infty)$, it is easy to construct a function $u\in W^{1,p}(\Omegai)$ such that $u\equiv 1$ in $f_1(\Omegai)$ and $u\equiv -1$ in $f_2(\Omegai)$. If $p>2$, $u$ cannot be extended to a function belonging to $W^{1,p}(\R^2)$ because the existence of such an extension  would contradict the  Sobolev imbedding of $W^{1,p}(\R^2)$ in $C(\R^2)$.
\\
In the case when $\rat=\rat^\star_\theta$, the situation depends in fact on the Hausdorff dimension of the set $\Xi=f_1(\Gamma)\cap f_2(\Gamma)$. The following extension theorem holds.
\begin{theorem}\label{ext-theorem}  see \cite{SEP, comparison}
\label{extension}
  Set $p^\star= 2-\dim_H \Xi$ (recall that $\Xi$ is defined in \eqref{Xi}).
  \begin{enumerate}
  \item If $p\in (1,p^\star)$, then
$\Omegai$ has the $W^{1,p}$-extension property.
\item If $p>p^\star$, then $\Omegai$ does not have the $W^{1,p}$-extension property.
  \end{enumerate}
\end{theorem}
\noindent
Point 1 in Theorem \ref{ext-theorem} was obtained in \cite{SEP}. Point 2 is a consequence of \cite{ADT} and \cite{comparison}: 
by Theorem \ref{trace-thm}, if $p>p^\star$, then ${W^{1,p}(\Omegai)}_{|\Gamma}\not\subset W^{1-\frac{2-d}p,p}(\Gamma) = {W^{1,p}(\R^2)}_{|\Gamma}$. This is in contradiction with the existence of a continuous extension operator from $W^{1,p}(\Omegai)$ to $W^{1,p}(\R^2)$ (see \cite{ADT} for the proof that the notions of traces coincide $\mu$-almost everywhere on $\Gamma$).

\begin{remark}\label{2-geometries}
As it was seen in \S \ref{dimension-gamma}, only two situations can occur, depending on the geometry of $\Omegai$:
\begin{list}{--}{\itemsep0mm \topsep1mm}
\item if $\theta$ is not of the form $\pi/(2k)$ for any integer $k$, then $\dim_H \Xi=0$ and $p^\star=2$,
\item if $\theta$ is of the form $\pi/(2k)$ for an integer $k$, then $\dim_H \Xi=(\dim_H \Gamma)/2$, and $p^\star=2- (\dim_H \Gamma)/2$.
\end{list}
\end{remark}

\begin{remark}\label{p=pstar}
The special case $p=p^\star$ is not dealt with in Theorem \ref{extension}. The latter is of particular importance in case 1 of Remark \ref{2-geometries} above,
 since the case $p=p^\star=2$ corresponds to the question of the $H^1$-extension property.
\\
In fact, it was proved by Koskela in \cite{MR1658090} that if a domain in $\R^n$ has the $W^{1,n}$-extension property, then it must have the $W^{1,p}$-extension property for every $p\geqs n$. Hence, a consequence of Theorem \ref{ext-theorem} is that $\Omegai$ cannot have the $W^{1,p}$-extension property when $p=2$. In particular, the domains that we will consider in Section \ref{r=rstar} fail to satisfy the $H^1$-extension property.
\\
To the best of our knowledge,  the question of the extension property for $p=p^\star$ in case 2 of Remark \ref{2-geometries} seems to be open.
\end{remark}

\section{The transmission problem in the case $r<r^\star_\theta$}\label{sec-transmission-pb}
\subsection{The transmission problem with fractal interface}\label{sec-pb-transmission}
The transmission problem  can be formally stated as
\begin{equation}\label{prob-transmission}
\left\{ \begin{array}{ll}
	-\Delta u = f & \mathrm{in}~ \Omegai\cup\Omegae,\\[1mm]
	[u]=0 &\text{on}~ \Gamma,\\[1mm]
	[\partial_n u]=\alpha u &\text{on}~ \Gamma,\\[1mm]
	\partial_n \ui =\partial_n \ue = 0 & \mathrm{on}~ \Sigma,\\[1mm]
	\partial_n \ue = 0,~~ \ui = u_0 & \mathrm{on}~ \Gamma^0,\\[1mm]
	\partial_n \ue = 0 & \mathrm{on}~ \partial D,
\end{array} \right.
\end{equation}
where $\alpha>0$, $\Sigma = \partial \Omegai\setminus (\Gamma\cup \Gamma^0)$, $[u]$ (resp. $[\partial_n u]$) denotes the jump of $u$ (resp. of the normal derivative of $u$) across $\Gamma$, $f\in L^2(D)$ and $u_0\in H^{1/2}(\Gamma^0)$. We also use the notations $\ui := u_{|\Omegai}$ and $\ue := u_{|\Omegae}$.
\\[1mm]
The transmission condition $[\partial_n u]=\alpha u$ on $\Gamma$ has no real meaning, since the normal is not defined on $\Gamma$. The rigorous meaning of \eqref{prob-transmission} is the following variational formulation:
\begin{equation}\tag{$P$}\label{prob-fv}
\begin{array}{c}
\text{find $u\in V$ such that for all $v\in V_0$,}\\[2mm]
	\displaystyle a(u,v) ~=~ \int_D f v \dx x,
\end{array}
\end{equation}
where $V$ is the affine space defined by
\begin{equation}\label{space-V}
V ~=~ \{u\in L^2(D),~ \ui\in H^1(\Omegai),~  \ue\in H^1(\Omegae),~ {\ui}_{|\Gamma^0}=u_0,~{\ui}_{|\Gamma} = {\ue}_{|\Gamma}\}.
\end{equation}
Recall that $v\in H^1(\Omegai) \mapsto v_{|\Gamma} \in H^{1/2}(\Gamma)$ and $v\in H^1(\Omegae) \mapsto v_{|\Gamma} \in H^{1/2}(\Gamma)$ are continuous maps, hence $V$ is closed. 
Note that if $u\in H^1(D)$ and $u_{|\Gamma^0}=u_0$ then $u\in V$.
\\
The vector space $V_0$ is defined as $V$, except that the condition ${\ui}_{|\Gamma^0}=u_0$ is replaced by ${\ui}_{|\Gamma^0}=0$. Finally,
\begin{equation}\label{form}
a(u,v) ~=~ \int_{\Omegai} \nabla \ui \cdot \nabla \vi  \dx x + \int_{\Omegae} \nabla \ue \cdot \nabla \ve \dx x + \alpha \int_\Gamma u_{|\Gamma} v_{|\Gamma} \dx \mu.
\end{equation}

\begin{remark}
The traces in the condition ${\ui}_{|\Gamma} = {\ue}_{|\Gamma}$ in \eqref{space-V} are meant in the sense of Definition \ref{def-trace}. The above definition of the space $V$ is suitable when $r<r^\star_\theta$ and in the special case when $r=r^\star_\theta$ and $\theta=0$ discussed in Section \ref{r=rstar}. In the other cases, the transmission condition has to be considerably changed, see Remark \ref{rem-theta>0}.
\end{remark}

\noindent
Note that the space $V_0$, equipped with the norm $a(u,u)^{1/2}$ is a Hilbert space.
From the Lax-Milgram theorem, we see that for every function $f$ given in $L^2(D)$, there exists a unique weak solution $u\in V$ to ($P$). 
Moreover, $u$ minimizes  the  functional
\begin{equation}\label{min-u}
v\in V\mapsto a(v,v) - 2\int_D f v \dx x.
\end{equation}

\subsection{The transmission problem with prefractal interface}
For any positive integer $n$, let us consider  the similar transmission problem in which  the interior domain has been truncated by stopping the construction at step $n$.
 This class of problems is much more standard since the interface $\Gamma^n$ consists of $2^n$ pairwise disjoint line segments.
The boundary value problem reads:
\begin{equation}\label{prob-transmission-n}
\left\{ \begin{array}{ll}
	-\Delta u = f & \mathrm{in}~ \Omegai^n\cup\Omegae^n,\\[1mm]
	[u]=0  &\text{on } \Gamma^n,\\[1mm]
	[\partial_n u]= \frac \alpha {|\Gamma^n|} u &\text{on } \Gamma^n,\\[1mm]
	\partial_n \ui^n =\partial_n \ue^n = 0 & \mathrm{on}~ \Sigma^n,\\[1mm]
	\partial_n \ue^n = 0,~~ \ui^n = u_0 & \mathrm{on}~ \Gamma^0,\\[1mm]
	\partial_n \ue^n = 0 & \mathrm{on}~ \partial D,
\end{array} \right.
\end{equation}
where $\Sigma^n = \partial \Omegai^n\setminus (\Gamma^n\cup \Gamma^0)$, and $[u]$ (resp. $[\partial_n u]$) denotes the jump of $u$ (resp. of the normal derivative of $u$) across $\Gamma^n$. We also use the notations $\ui^n = u_{|\Omegai^n}$ and $\ue^n = u_{|\Omegae^n}$.
\\[1mm]
The variational formulation of problem \eqref{prob-transmission-n}  can be stated as follows:
\begin{equation}\tag{$P_n$}\label{prob-n-fv}
\begin{array}{c}
	\text{find $u\in V^n$ such that for all $v\in V^n_0$,}\\[2mm]
	\displaystyle a_n(u,v) ~=~ \int_D f v \dx x,
\end{array}
\end{equation}
where $V^n$ is the affine space defined by
\begin{equation}\label{eq:7}
V^n ~=~ \{u\in L^2(D),~ \ui^n\in H^1(\Omegai^n),~ \ue^n\in H^1(\Omegae^n),~ {\ui^n}_{|\Gamma^0}=u_0,~ {\ui^n}_{|\Gamma^n} = {\ue^n}_{|\Gamma^n}\}.
\end{equation}

\begin{remark}
Let $G^n$ be the closure of the set $\partial \Omegai^n  \setminus (\Gamma^0\cup \Gamma^n)$, which is a finite union of polygonal lines. It is easy to see that $V^n$ is the set of the functions in $H^1(D\setminus G^n)$ such that ${\ui}_{|\Gamma^0}=u_0$.
\\
Similarly, if we define $G$ as $\partial \Omegai \setminus (\Gamma^0 \cup \Gamma)$, we observe that $G$ is not closed, since its closure contains $\Gamma$. Observe that, in general, the functions $u\in H^1(D\setminus \overline G)$ do not satisfy ${\ui}_{|\Gamma} = {\ue}_{|\Gamma}$, so $V$ cannot be identified with the set of the functions $u\in H^1(D\setminus \overline G)$ satisfying the Dirichlet boundary condition on $\Gamma^0$. On the other hand, since $D\setminus G$ is not an open set, dealing with $H^1(D\setminus G)$ is not very straightforward.
\end{remark}

Here also, $V^n_0$ is defined as $V^n$, except that the condition ${\ui^n}_{|\Gamma^0}=u_0$ is replaced by ${\ui^n}_{|\Gamma^0}=0$, and $a_n$ is defined by
\begin{equation}\label{form-n}
a_n(u,v) ~=~ \int_{\Omegai^n} \nabla \ui^n \cdot \nabla \vi^n ~ \dx x + \int_{\Omegae^n} \nabla \ue \cdot \nabla \ve^n \dx x + \frac \alpha {|\Gamma^n |}\int_{\Gamma^n} u_{|\Gamma^n} v_{|\Gamma^n} \dx x.
\end{equation}
The space $V^n_0$, equipped with the norm $a_n(u,u)^{1/2}$ is a Hilbert space. We also remark that $V^n\subset V$ with a continuous imbedding.
Again, the Lax-Milgram theorem implies that for every function $f$ given in $L^2(D)$, there exists a unique weak solution $u_n\in V^n$ to that problem, 
and $u_n$ minimizes  the  functional
\begin{equation}\label{min-u_n}
v\in V^n \mapsto a_n(v,v) - 2\int_D f v \dx x.
\end{equation}
\begin{lemma}
  \label{sec:transm-probl-with}
The sequence $u_n$ is bounded in $V$.
\end{lemma}
\begin{proof}
  Let $\tilde u_0$ be a function in $H^1(D)$ such that $\tilde u_{0| \Gamma^0}=u_0$, and such that $\tilde u _0$ is supported in 
a compact set which does not intersect the sets $\Gamma^n$, $\forall  n\ge 1$. 
 It is clear that $\tilde u_0\in V$ and that $\tilde u_0\in V_n$ for all $n\ge 1$.
Let us define $C_0=\int_D   \left(|\nabla \tilde u_0|^2  -2  f \tilde u_0\right)  \dx x$.
Thus, for all $n\ge 1$, 
\begin{equation}
  \label{eq:2}
a_n(u_n,u_n) -2 \int_D f u_n \dx x ~\leqs~  a_n(\tilde u_0,\tilde u_0) -2 \int_D f \tilde u_0 \dx x ~=~ C_0,
\end{equation}
because $\tilde u_{0|\Gamma^n}=0$.
On the other hand, since $V^n\subset V$, 
\begin{equation}
  \label{eq:3}
a_n(u_n,u_n) ~\geqs~ \int_{\Omegai^n}  |\nabla u_n|^2 \dx x   +\int_{\Omegae^n}   |\nabla u_n|^2 \dx x ~=~  
 \int_{\Omegai}   |\nabla u_n|^2 \dx x   +\int_{\Omegae}   |\nabla u_n|^2 \dx x.
\end{equation}
We shall also use the following Poincar{\'e} inequality: there exists a constant $C>0$ such that, for all $v\in V$,
\begin{equation}
  \label{eq:4}
  \|v\|^2 _{L^2 (D)} ~\leqs~ C \left(   \|v_{|\Gamma^0}\|^2_{L^2(\Gamma^0)}+  \int_{\Omegai}  |\nabla v|^2 \dx x   +\int_{\Omegae}   |\nabla v|^2 \dx x\right).
\end{equation}
From (\ref{eq:2}), (\ref{eq:3}) and (\ref{eq:4}), we deduce that
\begin{displaymath}
  \int_{\Omegai}   |\nabla u_n|^2 \dx x   +\int_{\Omegae}   |\nabla u_n|^2 \dx x- 2 \sqrt C  \|f\|_{L^2 (D)}
\left[   \|u_{0|\Gamma^0}\|^2_{L^2(\Gamma^0)}+  \int_{\Omegai}   |\nabla u_n|^2 \dx x   +\int_{\Omegae}   |\nabla u_n|^2 \dx x \right]^{\frac 1 2} \leqs 
C_0,
\end{displaymath}
which implies that the quantity $ \int_{\Omegai}   |\nabla u_n|^2 \dx x   +\int_{\Omegae}   |\nabla u_n|^2 \dx x$ is bounded by a constant independent of $n$. 
Using (\ref{eq:4}) again, this implies that $ \|u_n\| _{L^2 (D)}$ is also  bounded by a constant independent of $n$.
 Combining the previous two observations, we obtain that the sequence $u_n$ is bounded in $V$.
\end{proof}

\subsection{M-convergence of the energy forms in the case $\rat < \rat^\star_\theta$}\label{sec-convergence}
%


We start by extending the definition of the forms $a$ and $a_n$ to the whole space $L^2(D)$ by setting
\begin{eqnarray}
a(u,u) = \infty & \text{if } u\in L^2(D)\setminus V, \label{def-a}\\
a_n(u,u) = \infty & \text{if } u\in L^2(D)\setminus V^n. \label{def-a_n}
\end{eqnarray}

\noindent
We are interested in proving the convergence of the forms $a_n$ to $a$ in the following sense, introduced by Mosco (see \cite{MR1283033}).
\begin{definition}\label{Mosco-convergence}
A sequence of forms ${(a_n)}_{n}$ is said to M-converge to a form $a$ in $L^2(D)$ if
\begin{enumerate}[label=\textit{\roman*}.]
\item\label{point1} for every sequence ${(u_n)}_{n}$ weakly converging  to a function $u$ in $L^2(D)$,
\begin{equation} \label{M-conv1}
	\limi a_n(u_n,u_n) ~\geqs~ a(u,u)~~\mathrm{as}~n\to \infty,
\end{equation}
\item\label{point2} for every $u\in L^2(D)$, there exists a sequence ${(u_n)}_{n}$ converging strongly in $L^2(D)$ such that
\begin{equation} \label{M-conv2}
	\lims a_n(u_n,u_n) ~\leqs~ a(u,u)~~\mathrm{as}~n\to \infty.
\end{equation}
\end{enumerate}
\end{definition}


\begin{theorem}\label{M-cv}
Assume that $\rat < \rat^\star_\theta$, then the energy forms $a_n$ M-converge in $L^2(D)$ to the form $a$.
\end{theorem}

\begin{remark}
The $M$-convegence of forms differs from the $\Gamma$-convergence only in that the sequence $(u_n)$ in point \ref{point1} of Definition \ref{Mosco-convergence} is assumed to converge weakly instead of strongly. In the following, only the $\Gamma$-convergence of the energy forms $a_n$ will be needed.
\end{remark}


\begin{proof}
We will prove separately points \ref{point1} and \ref{point2} in Definition \ref{Mosco-convergence}.
\\[1mm]
\textit{Proof of point \ref{point1}}~ Suppose that ${(u_n)}$ weakly converges  to $u$ in $L^2(D)$. 
Without loss of generality, one can suppose $\limi a_n(u_n,u_n)$ is finite. We may further assume that there exists  a subsequence, still called $(u_n)$, such that 
 $a_n(u_n,u_n)$ converges to some real number as $n\to \infty$; as a consequence, 
there exists a constant  $c$ independent of $n$ such that
\begin{equation}\label{u-bound}
a_n(u_n,u_n) ~\leqs~ c.
\end{equation}
In particular, for all $n$, $u_n\in V^n$, which implies that $u_n\in V$.
Then, \eqref{u-bound} implies that $({u_n}_{|\Omegai})$ is bounded in $H^1(\Omegai)$, and $({u_n}_{|\Omegae})$ is bounded in $H^1(\Omegae)$. Therefore, there is a subsequence that we still note $(u_n)$ such that $({u_n}_{|\Omegai})$  converges weakly in $H^1(\Omegai)$, and strongly in $L^2(\Omegai)$. Since $u_n \rightharpoonup u$ in $L^2(D)$, we see that ${u_n}_{|\Omegai} \rhu \ui$ in $H^1(\Omegai)$. Similarly, up to  a further extraction of a subsequence, ${u_n}_{|\Omegae} \rhu \ue$ in $H^1(\Omegae)$. Consequently,
	$$\limi \left(\int_{\Omegai^n} {|\nabla u_n|}^2\dx x + \int_{\Omegae^n} {|\nabla u_n|}^2\dx x \right)
		~\geqs~ \int_{\Omegai} {|\nabla u|}^2\dx x + \int_{\Omegae} {|\nabla u|}^2\dx x.$$
We will now prove that $\frac 1{|\Gamma^n|}\int_{\Gamma^n} u_n^2\dx x \to \int_\Gamma u^2\dx \mu$ as $n\to \infty$, which will yield point \ref{point1}
\\
The following  inequality was proved in \cite{ach-tch-2007-1}:  for every $v\in H^1(\Omegai)$, $\left\| v-   \langle v \rangle_{\Gamma^0}  \right \|  _{L_\mu ^2 (\Gamma)} \leqs C   {\| \nabla v\|} _{L^2(\Omegai)}$. This implies that
\begin{displaymath}
    | \langle v \rangle_{\Gamma^0} |    ~\leqs~    \| v  \|  _{L_\mu ^2 (\Gamma)} + C   {\| \nabla v\|} _{L^2(\Omegai)}.
\end{displaymath}
Similarly,  $  \left\| v-   \langle v \rangle_{\Gamma^0}   \right\| _{L ^2 (\Gamma^0 )}   \leqs C   {\| \nabla v\|}_{L^2(\Omegai)}$ 
implies that
\begin{displaymath}
  \frac 1 {\sqrt { |\Gamma^0|} }\| v  \|  _{L^2(\Gamma^0)}       ~\leqs  ~   | \langle v \rangle_{\Gamma^0} |   + C   {\| \nabla v\|} _{L^2(\Omegai)}.
\end{displaymath}
for some constant independent of $v$ that we still note $C$. Combining these two inequalities, we obtain that  
\begin{displaymath}
\frac 1 {\sqrt { |\Gamma^0|} }\| v  \|  _{L^2 (\Gamma^0)}      ~ \leqs~ 
   \|v\|_{L_\mu ^2 (\Gamma)} +  C   {\| \nabla v\|}_{L^2(\Omegai)}.
\end{displaymath}
Hence, for every $\sigma\in \cA_n$,
\begin{eqnarray*}
\frac 1 {\sqrt { |\Gamma^\sigma|} }\| v  \|  _{L ^2 (\Gamma^\sigma)}       &\leqs &
   \|v\circ f_\sigma \|_{L_\mu ^2 (\Gamma)} +  C   {\| \nabla (v \circ f_\sigma) \|}_{L^2(\Omegai)}\\
&=& 2^{\frac n 2 }\|v \|_{L_\mu ^2 (f_\sigma(\Gamma))}+C   {\| \nabla v \|}_{L^2(f_\sigma(\Omegai))}.
\end{eqnarray*}
This yields that
\begin{displaymath}
  \frac 1 { |\Gamma^\sigma| }\| v  \|^2   _{L ^2 (\Gamma^\sigma)} \leqs  
2^{n  }\|v \|^2_{L_\mu ^2 (f_\sigma(\Gamma))}+ 2C 2^{\frac n 2 }  \|v \|_{L_\mu ^2 (f_\sigma(\Gamma))}  {\| \nabla v \|}_{L^2(f_\sigma(\Omegai))}
+ C^2   {\| \nabla v \|}^2 _{L^2(f_\sigma(\Omegai))}.
\end{displaymath}
Therefore
\begin{displaymath}
  \begin{split}
&\frac 1 {|\Gamma^n|} \int_{\Gamma^n} {u_n}^2 \dx x 
~= ~ \frac 1{2^n} \sum_{\sigma\in \cA_n} \frac 1 {|\Gamma^\sigma|} \int_{\Gamma^\sigma} {u_n}^2 \dx x\\
\leqs ~  & \|u_n \|^2 _{L_\mu ^2 (\Gamma)} +  2C 2^{-\frac n 2 }   \sum_{\sigma\in \cA_n}  
 \|u_n \|_{L_\mu ^2 (f_\sigma(\Gamma))}  {\| \nabla u_n \|}_{L^2(f_\sigma(\Omegai))}
+
\frac {C^2} {2^n}         \sum_{\sigma\in \cA_n}      {\| \nabla u_n \|}^2 _{L^2(f_\sigma(\Omegai))}\\
\leqs ~  &\|u_n \|^2 _{L_\mu ^2 (\Gamma)} +  2C 2^{-\frac n 2 }   \|u_n \| _{L_\mu ^2 (\Gamma)}  \left(   \sum_{\sigma\in \cA_n}  {\| \nabla u_n \|}^2 _{L^2(f_\sigma(\Omegai))}\right)^{\frac 1  2} +
\frac {C^2} {2^n}         \sum_{\sigma\in \cA_n}      {\| \nabla u_n \|}^2 _{L^2(f_\sigma(\Omegai))}.
  \end{split}
\end{displaymath}
Since  ${u_n}_{|\Omegai}$ is a bounded sequence in  $H^1(\Omegai)$, there exists a constant $M$ such that
\begin{displaymath}
  \frac 1 {|\Gamma^n|} \int_{\Gamma^n} {u_n}^2 \dx x  \leqs  \|u_n \|^2 _{L_\mu ^2 (\Gamma)} +  2C M 2^{-\frac n 2 }   \|u_n \| _{L_\mu ^2 (\Gamma)} +  
\frac {C^2 M^2} {2^n} . 
\end{displaymath}
Moreover, since $u_n$ weakly converges  to $u$ in $H^1(\Omegai)$, then up to the extraction of a subsequence, ${u_n}_{|\Gamma}$ strongly converges
  to $u_{|\Gamma}$ in $L^2_\mu(\Gamma)$, from  Remark~\ref{sec:trace-theorem-gamma}.
Hence, we obtain that
\begin{equation*}
\lims \frac 1 {|\Gamma^n|} \int_{\Gamma^n} {u_n}^2 \dx x ~\leqs ~ \int_{\Gamma} u^2 \dx \mu.
\end{equation*}
Similarly, the following inequality holds for every $v\in H^1(\Omegai)$:
\begin{equation}\label{poincare-ineq2}
{\| v\|}_{L^2_\mu(\Gamma)} ~\leqs ~ \frac 1 {\sqrt{|\Gamma^0|}} {\| v\|}_{L^2(\Gamma^0)} +  C' {\| \nabla v\|}_{L^2(\Omegai)}
\end{equation}
for some constant $C'$ independent of $v$. As above, we deduce  that
\begin{equation*}
\limi \frac 1 {|\Gamma^n|} \int_{\Gamma^n} {u_n}^2 \dx x ~\geqs ~ \int_{\Gamma} u^2 \dx \mu,
\end{equation*}
and we obtain the desired result.
\\[2mm]
\textit{Proof of point \ref{point2}}~ Take $u\in L^2(D)$. By \eqref{def-a}, we may assume that $u\in V$. We must construct $(u_n)$ converging strongly in $L^2(D)$ such that \eqref{M-conv2} holds. Note that the choice $u_n=u$ cannot be made, since $u\not\in V^n$ in general.
\\
Take $\delta>0$ and consider a neighborhood $\omega \subset D$ of $\Omegai$ such that $\Omegai \Subset \omega$ and $\sup_{x\in \omega}d(x,\Omegai)<\delta$, where $d(x,\Omegai) = \inf_{y\in \Omegai} |x-y|$. We introduce the notations $\omega^\sigma = f_\sigma(\omega)$ for  $\sigma \in \cA$ and $\omega_n = \bigcup_{\sigma\in \cA_n} \omega^\sigma$ for all integer $n$.
\\[1mm]
For every $n$, introduce the cut-off function $\chi_n$ in $D$ defined by
\begin{equation}
\chi_n(x) ~=~ {(1-\delta r^{-n} d(x,\omega_n))}^+,
\end{equation}
where $\alpha^+$ stands for the positive part of a real number $\alpha$.
Hence, $\chi_n \equiv 1$ in $\omega_n$ and $\chi_n \equiv 0$ outside $\tilde \omega_n := \{x\in D,~ d(x,\omega_n)<\delta r^n\}$. Note that if we set $\tilde \omega := \{x\in D,~ d(x,\omega)<\delta\}$ and $\tilde \omega^\sigma := f_\sigma(\tilde \omega) = \{x\in D,~ d(x,\omega^\sigma)<\delta r^n\}$ for $\sigma\in \cA_n$, then  $\tilde \omega_n = \bigcup_{\sigma\in \cA_n} \tilde \omega^\sigma$.
\\[1mm]
We can assume that $\delta$ is small enough so that $\tilde \omega^\sigma\cap \tilde \omega^\tau =\emptyset$ when $\sigma,\tau \in \cA_n$ and $\sigma\not=\tau$, since $f_\sigma(\Omegai)\cap f_\tau(\Omegai)=\emptyset$.
\\[1mm]
We now define a sequence of functions $u_n$ by
\begin{equation}
\label{u_n}
u_n = (1-\chi_n)u + \chi_n \cE (\ui),
\end{equation}
where $\cE$ is the extension operator introduced in \eqref{ext-op} and as above, $\ui = u_{|\Omegai}$.
Obviously, $u_n$ belongs to the space $V^n$ and the sequence $(u_n)$ strongly converges  to $u$ in $L^2(D)$. We will prove that $\lim a_n(u_n,u_n) = a(u,u)$ as $n\to \infty$. 
We start by showing that 
\begin{equation}\label{an->a-1}
	I_n :=\int_{\Omegai} {|\nabla u|}^2 \dx x + \int_{\Omegae} {|\nabla u|}^2 \dx x
	-\left( \int_{\Omegai^n} {|\nabla u_n|}^2 \dx x + \int_{\Omegae^n} {|\nabla u_n|}^2 \dx x\right)
	~\longrightarrow~ 0
\end{equation}
as $n\to \infty$. First observe that
\begin{equation*}
I_n ~=~ \int_{\tilde \omega_n\setminus \Omegai} {|\nabla u|}^2 \dx x - \int_{\tilde \omega_n\setminus \Omegai} {|\nabla u_n|}^2 \dx x.
\end{equation*}
Hence, it is enough to show that $\int_{\tilde \omega_n\setminus \Omegai} {|\nabla (u-u_n)|}^2\dx x \to 0$ as $n\to \infty$.
Note that
\begin{equation}
\int_{\tilde \omega_n\setminus \Omegai} {|\nabla(u-u_n)|}^2 \dx x~=~ \int_{\tilde \omega_n\setminus \Omegai} {|\nabla(\chi_n(\cE (\ui) -u))|}^2 \dx x
~\leqs~ 2(I_n^1 + I_n^2)
\end{equation}
where $I_n^1 = \int_{\tilde \omega_n\setminus \Omegai} {|\nabla(\cE (\ui) -u)|}^2 \dx x$ and
$I_n^2 = \int_{\tilde \omega_n\setminus \Omegai} {|(\nabla \chi_n)(\cE (\ui) -u)|}^2 \dx x$.
\\[1mm]
First observe that $I_n^1 \to 0$ as $n\to \infty$ since $\nabla (\cE (\ui) -u)\in L^2(\Omegae)$. We are left with dealing with $I_n^2$. One has
\begin{eqnarray*}
	I_n^2 &\leqs & \displaystyle  c\; r^{-2n} \int_{\tilde \omega_n\setminus \Omegai} {|\cE (\ui) -u|}^2\dx x\\
	&\leqs& \displaystyle c\; r^{-2n} \sum_{\sigma\in \cA_n} \int_{\tilde \omega^\sigma \setminus \Omegai} {|\cE (\ui) -u|}^2\dx x.
\end{eqnarray*}
where $c>0$ is a constant independent of $n$. Introduce the set $\hat \omega = {f_1}^{-1}(f_1(\tilde \omega)\setminus \Omegai)$. We have the following Poincar{\'e} inequality: there exists a constant $C$ such that for every $v\in H^1(\hat \omega)$ such that $v_{|\Gamma}=0$,
\begin{equation}\label{PW}
	\int_{\hat \omega} {|v|}^2\dx x ~\leqs~ C \int_{\hat \omega} {|\nabla v|}^2 \dx x.
\end{equation}
Observe that if $\delta$ is small enough, then $\tilde \omega^\sigma\setminus \Omegai = f_\sigma(\hat \omega)$ for every $\sigma\in \cA$. Therefore, applying a rescaled version of \eqref{PW} to the function $\cE (\ui)-u$, we obtain that there is a constant $c'>0$ indenpendant of $n$ such that
\begin{eqnarray*}
I_n^2 &\leqs& c'\sum_{\sigma\in \cA_n}  \int_{\tilde \omega^\sigma\setminus \Omegai} {|\nabla (\cE(\ui) - u)|}^2\dx x\\
 &=& c' \int_{\tilde \omega_n\setminus \Omegai} {|\nabla (\cE(\ui) - u)|}^2\dx x
 \end{eqnarray*}
since the sets $\tilde \omega^\sigma$, $\sigma \in \cA_n$ are pairwise disjoint.
We deduce that $I_n^2 \to 0$ as $n\to \infty$, since $\cE(\ui) - u\in H^1(\Omegae)$, which yields \eqref{an->a-1}.
\\
We will now prove that 
\begin{equation}\label{an->a-2}
	\frac 1 {|\Gamma^n|} \int_{\Gamma^n} {u_n}^2 \dx x \longrightarrow \int_\Gamma u^2 \dx \mu 
\end{equation}
as $n\to \infty$, which will conclude the proof of point \ref{point2}.
Observe that for every integer $n$, $\cE(\ui)_{|\Gamma^n} = u$, which implies that ${u_n}_{|\Gamma^n}=u$. 
We are left with proving that $\frac{1}{|\Gamma^n|} \int_{\Gamma^n} u^2\dx x \longrightarrow \int_{\Gamma} u^2\dx \mu $
as $n\to \infty$, which holds by Proposition \ref{convergence-result}. 
\end{proof}

%
A standard consequence of the Mosco-convergence of the energy forms proved in Theorem \ref{M-cv} is the convergence of the solutions 
of the problems \eqref{prob-n-fv} to the solution of problem \eqref{prob-fv}, in $L^2(D)$ and in $V$ (recall that $u_n$ is bounded in $V$ from Lemma~\ref{sec:transm-probl-with}). 

\begin{theorem}
Take $f\in L^2(D)$, and note $u_n$ (resp. $u$) the solution of problem \eqref{prob-n-fv} (resp. \eqref{prob-fv}).
The sequence $(u_n)$ converges to $u$ in the space $V$.
\end{theorem}

\section{A particular geometry with $\rat=\rat^\star_\theta$}\label{r=rstar}
\noindent
As seen before, the proof of Theorem \ref{M-cv} is based on the extension result of \S \ref{ext-subcritical}.
 In the case $\rat=\rat^\star_\theta$, the $H^1$-extension property is no longer true for the domain $\Omegai$ (see Remark \ref{p=pstar}).
 In what follows, we focus on the special case when $\theta=0$. We will see that in this case, the transmission condition imposed on $\Gamma$ yields an extension result (see Theorem \ref{ext-thm}) which is the main ingredient for proving the $M$-convergence of the energy forms.
\\[1mm]
In the case $\theta=0$, it can be seen that $r^\star_\theta =\frac 1 2$, and the ramified domain described in \S \ref{domains} is as in Figure \ref{trapeze}. In this particular case, the set $f_1(\Gamma)\cap f_2(\Gamma)$ 
is reduced to a single point that we call $A$. Observe that the self-similar part $\Gamma$ of the boundary is a line segment, and the self-similar measure $\mu$ associated with $\Gamma$ is the normalized one-dimensional Hausdorff measure.
\\
Since $\rat=\rat^\star_\theta$, the domain $\Omegae$ has infinitely many connected components. Call $U$ the outer connected component of $\Omegae$, \textit{ie} the only connected component which has a nonempty intersection with $\partial D$ (see Figure \ref{trapeze}).
Observe that $\Gamma$ is a subset of $\partial U$, and that the intersection of $\Gamma$ with the boundary of every other connected component of  $\Omegae$  is reduced to a single point. Apart from $U$, each connected component of $\Omegae$ is a triangle whose top vertex is at a dyadic point of $\Gamma$.
The largest triangle is named $T$, see Figure \ref{trapeze}, and all the other triangles are the images of $T$ by $f_\sigma$, $\sigma\in \cA_n$, $n\ge 1$.
\\
We consider the transmission problem
\begin{equation}\label{pb-transmission_2} 
\left\{ \begin{array}{ll}
	-\Delta u + \beta u= f & \mathrm{in}~ \Omegai\cup\Omegae,\\[1mm]
	{\ui}_{|\Gamma} = {\ue}_{|\Gamma} &    \\[1mm]
	[\partial_n u]=\alpha u &\text{on}~ \Gamma,\\[1mm]
	\partial_n \ui =\partial_n \ue = 0 & \mathrm{on}~ \partial \Omegai \setminus \Gamma,\\[1mm]
	\partial_n \ue = 0,~~ \ui = u_0 & \mathrm{on}~ \Gamma^0,\\[1mm]
	\partial_n \ue =  0 & \mathrm{on}~ \partial D,
\end{array} \right.
\end{equation}
where $\alpha$ and $\beta$ are positive numbers. The trace ${\ue}_{|\Gamma}$ in the transmission condition
is meant as the trace of the function $u_{|U}$ on the set $\Gamma$.
\begin{remark}
  \label{sec:part-geom-with}
The reason for considering the operator $-\Delta u+\beta u$ with $\beta>0$ instead of $-\Delta u$ as in the former case is that, in the present case,
 $\Omegae$ is an infinite union of disjoint connected sets: $\Omegae=U\cup \bigcup_{\sigma\in \cA}  f_\sigma(T)$. Therefore, (\ref{pb-transmission_2}) 
involves Neumann problems in $T$ and in $ f_\sigma(T)$, $\sigma\in \cA_n$, $n\ge 1$, which are not  well posed if $\beta=0$ and 
the average of $f$ in these sets is not zero. It would also be possible to consider the case $\beta=0$ under additional assumptions, on the support of $f$ for example, but 
this would imply further technical details, because the solutions of Neumann problems in the holes would then be defined up to the addition of constants.
\end{remark}

\begin{remark}\label{rem-theta>0}
When $\theta>0$ and $r=r^*$, the situation is quite different: $\Gamma$ is not entirely contained in the boundary of any connected component of $\Omegae$. 
It can be shown that there exists $\delta\in (1,d)$ such that the intersections of $\Gamma$ with the boundary of the connected components of $\Omegae$ 
have Hausdorff dimension $\delta$. These sets are called the canopies of the domain $\Omegai$.
 In this case, the transmission condition has to be described more carefully. This is the topic of a work in progress.
\end{remark}
The meaning of (\ref{pb-transmission_2}) is the  variational formulation (\ref{prob-fv}) where $V$ is defined by (\ref{space-V}) and 
\begin{equation}\label{eq:6}
a(u,v) ~=~ \int_{\Omegai} ( \nabla \ui \cdot \nabla \vi +\beta \ui \vi) ~ \dx x + \int_{\Omegae} (\nabla \ue \cdot \nabla \ve + \beta \ue\ve)~\dx x   + \alpha \int_\Gamma u_{|\Gamma} v_{|\Gamma} \dx \mu.
\end{equation}
In order to set the transmission problems in the geometries with prefractal interfaces, we first need to define some trapezoidal subsets of the triangular holes as follows:
let $H$ be the height of the triangle $T$. Choosing the coordinates in such a way that $\Gamma^0$ is a segment of the line $\{x_2=0\}$,
 we see that $\Gamma$ is a segment of the line $\{x_2=2H\}$. Then, we can also define  $\widehat T^n$ and $ \widehat T_\sigma  ^n$ by 
 $\widehat T^n = T\cap \{    (2-   3 /2^{n+1}) H < x_2< (2- 2^{-n}) H \}$ and 
$ \widehat T_\sigma  ^n=  f_\sigma(T)\cap \{    (2-   3 /2^{n+1}) H < x_2< (2- 2^{-n}) H \}$
for $\sigma\in \cA_m$ and $m<n$.
\\  
 Finally we define 
 \begin{equation}
   \label{eq:1}
\widehat \omega^n_{\rm ext}=   \bigcup_{m=0}^{n-1} \bigcup_{\sigma\in \cA_m}  \widehat T_\sigma  ^n   
 \;\; \subset \;  \Omegae^n \cap \bigg\{    2H-   \frac 3 {2^{n+1}} H < x_2< 2H- 2^{-n} H \bigg\}.
 \end{equation}
The transmission problem with interface $\Gamma^n$ is then 
\begin{equation}\label{prob-transmission-2n}
\left\{ \begin{array}{ll}
	- {\rm div } \left (  \nu_n   \nabla u\right)   +\beta u     = f & \mathrm{in}~ \Omegai^n\cup\Omegae^n,\\[1mm]
	[u]=0  &\text{on } \Gamma^n,\\[1mm]
	[\partial_n u]= \frac \alpha {|\Gamma^n|} u &\text{on } \Gamma^n,\\[1mm]
	\partial_n \ui^n =\partial_n \ue^n = 0 & \mathrm{on}~ \Sigma^n,\\[1mm]
	\partial_n \ue^n = 0,~~ \ui^n = u_0 & \mathrm{on}~ \Gamma^0,\\[1mm]
	\partial_n \ue^n = 0 & \mathrm{on}~ \partial D,
\end{array} \right.
\end{equation}
where
\begin{equation}
  \label{eq:5}
\nu_n = 2^{-2n} \one_{ \widehat \omega^n_{\rm ext}     }  + \one_{D\backslash \widehat \omega^n_{\rm ext}}.
\end{equation}
Note that $\nu_n =1$ in $\Omegai^n$ and that the Lebesgue measure of the set where $\nu_n = 2^{-2n}$ vanishes as $n\to \infty$.
Hence $\nu_n$ tends to $1$ almost everywhere in $D$. 
The variational formulation of (\ref{prob-transmission-2n}) is (\ref{prob-n-fv}) with $V_n$ defined in (\ref{eq:7}), and $a_n$ defined as follows:
\begin{equation}\label{eq:8}
  \begin{split}
&a_n(u,v)\\ =&  \int_{\Omegai^n} (\nabla \ui^n \cdot \nabla \vi^n  +\beta   \ui^n \vi^n) ~ \dx x + \int_{\Omegae^n}  (\nu_n  \nabla \ue^n  \cdot \nabla \ve^n + \beta \ue^n   \ve^n )\dx x + \frac \alpha {|\Gamma^n |}\int_{\Gamma^n} u_{|\Gamma^n} v_{|\Gamma^n} \dx x.    
  \end{split}
\end{equation}
\noindent
\begin{remark}
  \label{sec:part-geom-with-1}
The reason for modifying the partial differential equation in (\ref{prob-transmission-2n}) by taking $-{\rm div } \left (  \nu_n   \nabla u\right)$ instead of  $-\Delta u$ in 
(\ref{prob-transmission-n}) is that for a function $u\in V $, $u_{|\Omegai}$  is completely independent from $u_{|f_\sigma(T)}$.
 This explains why  the construction of a sequence of functions $(u_n)$  such that   $u_n\in V_n$, $u_n\to u$ in $L^2(D)$ and $a_n(u_n, u_n)\to a(u,u)$,
 is difficult without modifying the coefficients of the partial differential equation near the top of the triangles $T$ and $f_\sigma(T)$ in order to cope with  the 
possibly strong gradients of $u_n$. Although we have not tried it, it may be possible to choose a parameter larger than $2^{-2n}$ in the definition of $\nu_n$.
\end{remark}
\begin{figure}[h!]
\begin{center}
\scalebox{0.4}{\input{trapeze2.pstex_t}}
\caption{An exemple of the domains $\Omegai$ and $\Omegae$ in the case $\theta=0$}\label{trapeze}
\end{center}
\end{figure}
\noindent
The main result of this paragraph is the following theorem
\begin{theorem}\label{M-cv-theta=0}
Assume that $\theta=0$ and $\rat=\rat^\star_\theta=\frac 1 2$. Then the energy forms $a_n$ defined in (\ref{eq:8}) M-converge in $L^2(D)$ to the form $a$ defined in (\ref{eq:6}).
\end{theorem}
\noindent
Since  $\partial U$ is Lipschitz-continuous, a standard trace result yields that for every $u\in V$, ${\ue}_{|\Gamma} \in H^{1/2}(\Gamma)$. 
Hence, the transmission condition in (\ref{pb-transmission_2}) implies that
\begin{equation}\label{trace-theta=0}
{\ui}_{|\Gamma}\in H^{1/2}(\Gamma).
\end{equation}
Note that (\ref{trace-theta=0}) is not only a consequence of the fact that $\ui \in H^1(\Omegai)$, because the latter property only implies that  ${\ui}_{|\Gamma}$ in $H^s(\Gamma)$ for all $s<\frac 1 2$ (see Theorem \ref{trace-thm}).
\\[1mm]
For proving Theorem \ref{M-cv-theta=0}, we need the following extension result, which is not available in the literature:
\begin{theorem}\label{ext-thm}
There exist a linear extension operator $\cF$ from $\{v\in H^1(\Omegai),~ v_{|\Gamma}\in H^{1/2}(\Gamma)\}$ to $H^1(D)$ and a constant $C>0$ such that for every $v\in H^1(\Omegai)$ with $v_{|\Gamma}\in H^{1/2}(\Gamma)\}$,
\begin{equation}\label{continuity}
{\Vert \tilde v \Vert}_{H^1(D)}^2 ~\leqs~ C \left( {\Vert v \Vert}_{H^1(\Omegai)}^2 + {\Vert v_{|\Gamma} \Vert}_{H^{1/2}(\Gamma)}^2 \right).
\end{equation}
\end{theorem}


Theorems \ref{auscher-badr} and \ref{lem-PT} below will play an important role in the proof of Theorem \ref{ext-thm}. We start by recalling an extension result for multiple cones from \cite{MR2790816}.

\begin{theorem}\label{auscher-badr}[see \cite{MR2790816}, Th. 5.1]
Call $C$ the double cone in $\R^2$ defined by $|x_1| <|x_2|$. 
Write $\rho(x)=\Vert x \Vert$ (the notation $\Vert \; . \;  \Vert$ stands for the euclidean norm), and, for every $v\in H^1(C)$, introduce the antiradial part $v_a$ of $v$ in the cone $C$, defined by
	$$v_a(x) ~=~ v(x) -  {\langle v \rangle}_{S_{\rho(x)}},$$
where for any $R>0$, $S_R=\{x\in C,~ \Vert x \Vert =R\}$, and ${\langle v \rangle}_{S_R}$ is the mean value of $v$ along the arc $S_R$.
\\
There exists a linear extension operator 
\begin{equation}
	\Lambda : \left\{v\in H^1(C),~ \frac{v_a}\rho \in L^2(C)\right\} \to H^1(\R^2),
\end{equation}
such that for every $v\in \{v\in H^1(C),~ {v_a}/\rho \in L^2(C)\}$,
$${\Vert \Lambda v \Vert}_{H^1(\R^2)} ~\leqs ~ c\left( {\Vert v \Vert}_{H^1(C)} + {\left\Vert \frac {v_a} \rho \right\Vert}_{L^2(C)} \right)$$
where $c>0$ is a constant independent of $v$.
\end{theorem}
\noindent

\begin{remark}\label{rem-AB}
The construction in \cite{MR2790816} is such that if $v$ is radial (\textit{resp.} constant) in $C\cap B(0,R)$, then $\Lambda v$ is radial (\textit{resp.} constant) in $B(0,R)$.
\end{remark}

As mentioned in \cite{MR2790816}, Theorem \ref{auscher-badr} can be immediately extended in $\R^n$ to the case of a union of two half-cones sharing the same vertex, separated by a hyperplane passing through the vertex and not containing any direction of the boundaries.

\begin{theorem}[Peetre-Tartar]\label{lem-PT}[see \cite{MR0221282,Tartar}]
Let $E,E_1,E_2,F$ be Banach spaces, and let $A_i$, $i=1,2$, be continuous linear operators from $E$ to $E_i$, and suppose $A_1$ is compact. Further assume that there exists a constant $c_0>0$ such that for any $v\in E$,
\begin{equation}\label{cond-PT}
	{\Vert v\Vert}_E ~\leqs ~ c_0({\Vert A_1 v\Vert}_{E_1} + {\Vert A_2 v\Vert}_{E_2}).
\end{equation}
If $L$ is a coninuous linear operator from $E$ to $F$ such that $L_{|\ker A_2} \equiv 0$, then there exists a constant $c_1>0$ such that for any $v\in E$,
\begin{equation}\label{PT}
	{\Vert Lv \Vert}_F ~\leqs~ c_1 {\Vert A_2 v \Vert}_{E_2}.
\end{equation}
\end{theorem}

\paragraph{Notations}
We start by introducing notations for the proof of Theorem \ref{ext-thm}.
\\[1mm]
We first introduce a domain $C$ which is the union of two truncated half-cones included in $\Omegai$, whose common vertex is the point $A$.
Recall that $T$ is the main hole of the domain $\Omegai$. Call $\ph_0\in (0,\frac \pi 2)$ the upper half-angle of the triangular domain $T$ (see Figure \ref{cone}), and take $\ph_1>\ph_0$. Call $\mathcal{C}$ the half-cone whose boundary is made of the two half-lines through $A$ with respective angles $\ph_0$ and $\ph_1$ with the vertical axis (see Figure \ref{cone}). Call $C_2 = \mathcal{C}\cap (\tildeOmegai\setminus \overline{Y^0})$. We can assume that $\ph_1>\ph_0$ is small enough so that $C_2\subset \Omegai$, in other words $C_2$ does not intersect any of the holes of $\Omegai$. We define $C_1$ to be the symmetric of $C_2$ with respect to the vertical axis $x_2=0$, and we write $C=C_1\cup C_2$.
%
\\[1mm]
We also introduce the sets $Y^{k,1} := f_1\circ {f_2}^{k-1}(Y^0)$ and $Y^{k,2} := f_2\circ {f_1}^{k-1}(Y^0)$ for every $k\geqs 1$ (see Figure \ref{cone}), and we write $Y^k := Y^{k,1}\cup Y^{k,2}$. We also note $\gamma := {f_1}^2(\Gamma)\cup {f_2}^2(\Gamma)$ (see Figure \ref{Yki}).

\begin{figure}[h!]
   \centering\[
\scalebox{0.55}{\input{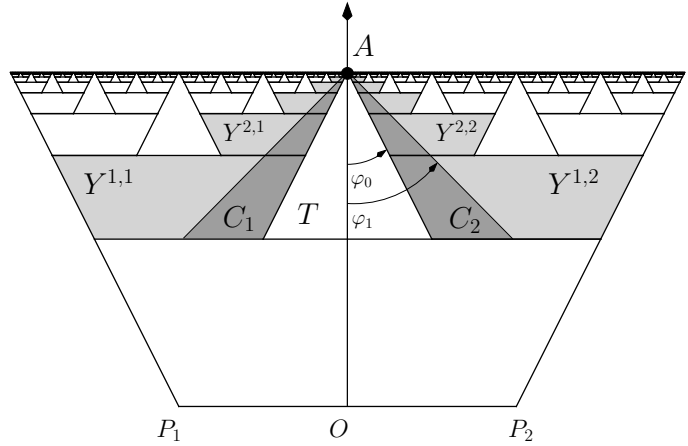}}
  \]
   \caption{The region $\cC=C_1\cup C_2$ and the domains $Y^{k,i}$, $k,i=1,2$.}
   \label{cone}
 \end{figure}

\noindent
Call $\Omega^{1,1}:=f_1\big(\Omegai\setminus \overline{f_2(\Omegai)}\big)$ and $\Omega^{1,2}:=f_2\big(\Omegai\setminus \overline{f_1(\Omegai)}\big)$ (see Figure \ref{Yki}).
We introduce the sets $\Omega^{k,i}$ defined by $\Omega^{k,i}:=g^{k-1}(\Omega^i)$, $k\geqs 1$, $i=1,2$, where $g$ is the homothety centered at $A$ with ratio $1/2$, see Figure \ref{Yki}.

\begin{figure}[h!]
\begin{textblock}{6}(8.5,2) $\gamma = {f_1}^2(\Gamma)\cup {f_2}^2(\Gamma)$ \end{textblock}
   \centering\[
\scalebox{0.5}{\input{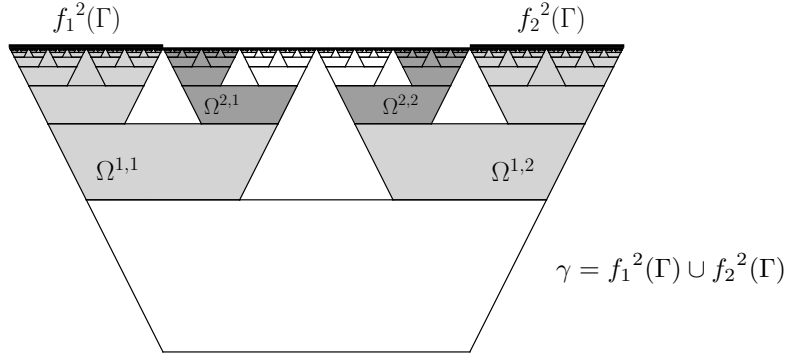}} 
  \]
   \caption{The domains $\Omega^{k,i}$, $k,i=1,2$ (in light grey: $\Omega^{1,1}$ and $\Omega^{1,2}$, in dark grey: $\Omega^{2,1}$ and $\Omega^{2,2}$) and the sets ${f_i}^2(\Gamma)$, $i=1,2$.}
   \label{Yki}
 \end{figure}

\noindent
For every integer $k\geqs 1$, we introduce $\Sigma^k := \{\tau\in \cA,~ f_\tau(Y^0) \subset \Omega^{k,1}\cup\Omega^{k,2}\}$, with the notations of \S \ref{def}.
\\[1mm]
Take $\kappa\in (1/2,1)$. We introduce the space $G=\{v\in L^1_\text{loc}(\Omega^{1,1}\cup \Omega^{1,2}),~ {\Vert v \Vert}_G<\infty\}$, where
\begin{equation}\label{normG}
{\Vert v \Vert}_G^2 = \sum_{m\geqs 1} \kappa^m \sum_{\tau \in \Sigma^1\cap \cA_m} \int_{f_\tau(Y^0)} {|v |}^2 \dx x.
\end{equation}
Endowed with the norm ${\Vert \cdot \Vert}_G$, the space $G$ is a Hilbert space.
\\[1mm]
We also introduce the space $H=\{v\in L^2(\Omega^{1,1}\cup \Omega^{1,2}),~ \nabla v\in G\}$, which is a Hilbert space with the norm
\begin{displaymath}
   \left(  {\|\nabla v\|}^2 _G +  {\|v\|}^2 _{L^2(\Omega^{1,1}\cup \Omega^{1,2})}  
   \right)^{\frac 1 2}.
\end{displaymath}
Moreover, from Theorem~\ref{SPI}, we see that $v\in H\mapsto v_{|\gamma}$ is a continuous operator from 
$H$ to $L^2_\mu (\gamma)$.  Arguing by contradiction, we can  show that 
\begin{displaymath}
  \left(  {\|\nabla v\|}^2_G +  {\|v_{|\gamma}\|}^2 _{L^2 _\mu (\gamma)}  \right)^{\frac 1 2}
\end{displaymath}
is an equivalent norm on $H$.
\\[2mm]
We first state and prove two lemmas which will be useful in the proof of Theorem \ref{ext-thm}.

\begin{lemma}\label{lem1}
There exists a constant $c>0$ such that for every $v\in H$,
\begin{equation}
\int_{Y^1} \left({|v(x)-{\langle v \rangle}_{Y^{1,1}}|}^2 + {|v(x)-{\langle v \rangle}_{Y^{1,2}}|}^2\right)\dx x
~\leqs~ c\left(\int_\gamma {|v-{\langle v \rangle}_\gamma|}^2\dx \mu + {\Vert \nabla v\Vert}^2_G \right). \label{eq-pt}
\end{equation}
\end{lemma}

\begin{proof}
We introduce the Hilbert space $E = \{(v,w)\in H\times L^2_\mu(\gamma),~ v_{|\gamma} = w\}$, endowed with the norm ${\Vert \cdot \Vert}_E$ given by ${\Vert (v,w)\Vert}_E^2 ={\Vert \nabla v \Vert}_G^2 + {\Vert w \Vert}_{L^2_\mu(\gamma)}^2$. 
\\[1mm]
We now introduce the operators 
\begin{eqnarray*}
&&A_1 : (v,w)\in E \mapsto {\langle v \rangle}_{\gamma},\\[1mm]
&&A_2 : (v,w)\in E \mapsto (\nabla v, w-{\langle w \rangle}_\gamma),\\[1mm]
&&L : (v,w)\in E \mapsto (v-{\langle v \rangle}_{Y^{1,1}}, v-{\langle v \rangle}_{Y^{1,2}}).
\end{eqnarray*}
With the notations of Theorem \ref{lem-PT}, $E_1=\R$, $E_2=G \times L^2_\mu(\gamma)$ and $F={L^2(Y^{1})}^2$. It is easily seen that $A_1$, $A_2$ and $L$ are continuous linear operators, and $A_1$ is compact. Moreover, \eqref{cond-PT} is clearly satisfied.
\\
Observe that $(v,w)\in E$ lies in $\ker A_2$ if and only if $v$ is constant in $\Omega^{1,1}$ and in $\Omega^{1,2}$. Hence, it is obvious that $L_{|\ker A_2}\equiv 0$. From this, we deduce by Theorem \ref{lem-PT} that there exists a constant $c>0$ such that ${\Vert L(v,w) \Vert}_F ~\leqs~ c {\Vert A_2 (v,w) \Vert}_{E_2}$ for all $(v,w)\in E$, which yields \eqref{eq-pt}.
\end{proof}

\begin{lemma}\label{lem2}
Assume that $v\in H^1(\Omegai)$ and $v_{|\Gamma}\in H^{1/2}(\Gamma)$, then $\frac {v_a}\rho \in L^2(C)$, and
\begin{eqnarray*}
\int_C {\left| \frac {v_a}\rho \right|}^2 \dx x 
&\leqs & c \left(  \sum_{k\geqs 1} 2^k \int_{\gamma_k} {|v-{\langle v \rangle}_{\gamma_k}|}^2\dx \mu 
+ \sum_{k\geqs 1} \sum_{p\geqs k} \kappa^{p-k+1} \sum_{\tau \in \Sigma^k\cap \cA_p} \int_{f_\tau(Y^0)} {|\nabla v|}^2\dx x \right) \label{eq-lem2}
\end{eqnarray*}
for some constant $c>0$ independent of $v$, where $\gamma^k=g^{k-1}(\gamma)$ (recall that $g$ is the homothety centered at $A$, with ratio $1/2$).
\end{lemma}

\begin{proof}
We first observe that, by self-similarity, there exists a constant $c_1>0$ such that for all $x\in Y^k$, $\rho(x)\geqs \frac {c_1}{2^k}$. Therefore,
$$\int_\cC {\left|{\frac{v_a}{\rho}}\right|}^2 \dx x ~\leqs~ c_1 \sum_{k\geqs 1} 2^{2k} \int_{Y^k} {|v_a|}^2 \dx x,$$
since $\cC \subset \bigcup_{k\geqs 1} Y^k$ by construction. Hence, there is a constant $c_2>0$ such that $\ds \int_\cC {\left|\frac{v_a}\rho\right|}^2\dx x \leqs c_2(I_1+I_2)$, where
\begin{eqnarray}
I_1 &=& \sum_{k\geqs 1} 2^{2k}\int_{Y^k} {|v(x)-{\langle v \rangle}_{Y^k}|}^2 \dx x,\\
I_2 &=& \sum_{k\geqs 1} 2^{2k}\int_{Y^k}  {|{\langle v \rangle}_{Y^{k}} -{\langle v\rangle}_{S_{\rho(x)}}|}^2 \dx x.
\end{eqnarray}
We start by dealing with $I_1$. We note that
\begin{equation}
I_1 ~\leqs~ \frac 1 2 \sum_{k\geqs 1} 2^{2k} \int_{Y^{k}} \left({|v-{\langle v \rangle}_{Y^{k,1}}|}^2 +{|v-{\langle v \rangle}_{Y^{k,2}}|}^2\right) \dx x.
\end{equation}
For every $k\geqs 1$, we can apply Lemma \ref{lem1} to the function $v\circ g^{k-1}$. Since $\gamma^k=g^{k-1}(\gamma)$, we obtain
\begin{eqnarray}
I_1 &\leqs & c\left( \sum_{k\geqs 1} 2^k \int_{\gamma_k} {|v-{\langle v \rangle}_{\gamma_k}|}^2\dx \mu
+ \sum_{k\geqs 1} \sum_{m\geqs 1} \kappa^{m} \sum_{\tau \in \Sigma^1\cap \cA_m} \int_{g^{k-1}(f_\tau(Y^0))} {|\nabla v|}^2\dx x \right) \notag \\
&=& c\left( \sum_{k\geqs 1} 2^k \int_{\gamma_k} {|v-{\langle v \rangle}_{\gamma_k}|}^2\dx \mu +
\sum_{k\geqs 1} \sum_{p\geqs k} \kappa^{p-k+1} \sum_{\tau \in \Sigma^k\cap \cA_p} \int_{f_\tau(Y^0)} {|\nabla v|}^2\dx x \right), \label{ineq-I1}
\end{eqnarray}
for some constant $c>0$ independent of $v$, since $\{g^{k-1}\circ f_\tau,~ \tau \in \Sigma^1\cap \cA_m\} = \{f_\tau,~ \tau \in \Sigma^k\cap \cA_{m+k-1}\}$. 
\\
Let us now deal with $I_2$.
For every $R>0$ and $i=1,2$, call $S_R^i = S_R \cap \cC_i$ and ${\langle v \rangle}_{S_R^i}$ the mean value of $v$ on the set $S_R^i$. We note that 
\begin{equation}\label{maj-I2}
	I_2 ~\leqs~ \frac 1 2 \sum_{i=1,2} \; \sum_{k\geqs 1} 2^{2k} \left(
	\int_{Y^{k}} {|{\langle v \rangle}_{Y^{k,i}} - {\langle v \rangle}_{S_{\rho(x)}^i}|}^2 \right).
\end{equation}
 Take $i\in \{1,2\}$ and $x_0\in f_i(Y^0)$, and, for every integer $k\geqs 1$,  $\rho_k = \rho(x_0)/2^{k-1}$.
\\[1mm]
We observe that $\sum_k 2^{2k} \int_{Y^{k}} {|{\langle v \rangle}_{Y^{k,i}} - {\langle v \rangle}_{S_{\rho(x)}^i}|}^2 \leqs 2(J_1+J_2)$, where
\begin{eqnarray}
	J_1 &=& \sum_{k\geqs 1} 2^{2k+1} |Y^{k,i}|  {({\langle v \rangle}_{Y^{k,i}} - {\langle v \rangle}_{S_{\rho_k}^i})}^2, \label{J1}\\
	J_2 &=& \sum_{k\geqs 1} 2^{2k} \int_{Y^{k}} {|{\langle v \rangle}_{S_{\rho_k}^i} - {\langle v \rangle}_{S_{\rho(x)}^i}|}^2 \dx x.
\end{eqnarray}
Let us first examine $J_1$. The following Poincar{\'e} inequality holds in $Y^{1,i}=f_i(Y^0)$: for every $v\in H^1(Y^{1,i})$,
\begin{equation}\label{poincare}
	\int_{Y^{1,i}} {|v(x)-{\langle v \rangle}_{S_{\rho_i}}|}^2 \dx x ~\leqs~ M \int_{Y^{1,i}} {|\nabla v|}^2 \dx x,
\end{equation}
where the constant $M>0$ is independent of $v$.
\\[1mm]
Observe that for every integer $k\geqs 1$, $Y^{k,i} = g^{k-1}(Y^{1,i})$, and $S^i_{\rho_k} = g^{k-1}(S^i_{\rho_1})$. 
Then
\begin{eqnarray*}
{|{\langle v \rangle}_{Y^{k,i}} - {\langle v \rangle}_{S_{\rho_k}^i}|}^2
&=& \frac{1}{{|Y^{k,i}|}^2}  {\left|\int_{Y^{k,i}} (v(x)-{\langle v \rangle}_{S_{\rho_k}^i})\dx x \right|}^2\\
&\leqs& \frac 1 {|Y^{k,i}|} \int_{Y^{k,i}} {|v(x)-{\langle v \rangle}_{S_{\rho_k}^i}|}^2 \dx x\\
&\leqs& \frac{1}{2^{2(k-1)}|Y^{k,i}|} \int_{Y^{1,i}} {|v\circ g^{k-1} (x)-{\langle v\circ g^{k-1} \rangle}_{S_{\rho_1}^i}|}^2 \dx x\\
&\leqs& \frac{M}{2^{2(k-1)}|Y^{k,i}|} \int_{Y^{1,i}} {|\nabla(v\circ g^{k-1})|}^2\dx x\\
&=& \frac{M}{2^{2(k-1)}|Y^{k,i}|} \int_{Y^{k,i}} {|\nabla v|}^2\dx x,
\end{eqnarray*}
where we used \eqref{poincare}. From this and \eqref{J1},  we deduce  that 
\begin{equation}\label{ineq-J1}
J_1~\leqs ~ 8M \sum_{k\geqs 1} \int_{Y^{k,i}} {|\nabla v |}^2.
\end{equation}
To deal with $J_2$, we use polar coordinates $(\rho,\ph)$ centered at $A$. Introduce the positive constants $R_0,R_1$ such that for all  $(\rho,\ph)\in Y^{1,i}$, $R_0\leqs \rho\leqs R_1$. Observe that if $x\in Y^{k}$,
\begin{eqnarray*}
	{|{\langle v \rangle}_{S_{\rho(x)}^i} -{\langle v \rangle}_{S_{\rho_k}^i} |}^2 
	&=& {\left|\frac{1}{\ph_1-\ph_0} \int_{\ph_0}^{\ph_1} (v(\rho(x),\ph)-v(\rho_k,\ph)) \dx\ph\right|}^2\\
	&=& {\left| \frac{1}{\ph_1-\ph_0} \int_{\ph_0}^{\ph_1} \int_{\rho_k}^{\rho(x)} \frac{\partial v}{\partial \rho}(s,\ph)\dx s \dx\ph \right|}^2\\
	&\leqs& \frac{1}{(\ph_1-\ph_0)^2} \left|\int_{\ph_0}^{\ph_1} \int_{\rho_k}^{\rho(x)} {\left|\frac{\partial v}{\partial \rho}(s,\ph)\right|}^2 s \dx s \dx\ph\right| \times 
	\left|\int_{\ph_0}^{\ph_1} \int_{\rho_k}^{\rho(x)} \frac{\dx s}s \dx\ph\right|\\
	&\leqs& \frac{\log \frac {R_1}{R_0}}{\ph_1-\ph_0} \int_{\ph_0}^{\ph_1} \int_{R_0/2^{k-1}}^{R_1/2^{k-1}} {\left|\frac{\partial v}{\partial \rho}(s,\ph)\right|}^2 s \dx s \dx \ph\\
	&=& \frac{\log \frac {R_1}{R_0}}{\ph_1-\ph_0} \int_{C(k)} {|\nabla v(y)|}^2 \dx y,
\end{eqnarray*}
where $C(k) = \{(\rho,\ph),~ \frac{R_0}{2^{k-1}} < \rho < \frac{R_1}{2^{k-1}},~ \ph_1<|\ph|<\ph_0\}$.
%
%
Hence,
$$J_2 ~\leqs~ \frac{\log \frac {R_1}{R_0}}{\ph_1-\ph_0} \sum_{k\geqs 1} 2^{2k} \int_{Y^{k}}  \int_{C(k)} {|\nabla v(y)|}^2 \dx y \dx x,$$
and there is a constant $c_3>0$ independent of $v$ such that
$$J_2 ~\leqs~ c_3 \sum_{k\geqs 1} \int_{C(k)} {|\nabla v(y)|}^2 \dx y.$$
Note that every point $(\rho,\ph)$ in $C(k)$ lies in at most $\log_2 \frac {R_1}{R_0}$ sets $C(l)$, $l\geqs 1$. Therefore,
\begin{equation}\label{ineq-J2}
J_2 ~\leqs~ c_3 \log_2 \frac {R_1}{R_0}  \int_{C} {|\nabla v(y)|}^2\dx y ~\leqs ~ c_3 \log_2 \frac {R_1}{R_0} \sum_{k\geqs 1} \int_{Y^k} {|\nabla v|}^2\dx x.
\end{equation}
By the inequalities \eqref{ineq-J1} and \eqref{ineq-J2}, $I_2 \leqs c_4 \sum_{k\geqs 1} \int_{Y^k} {|\nabla v|}^2\dx x$ for some constant $c_4>0$ independent of $v$. Hence, 
\begin{equation}\label{ineq-I2}
I_2 ~\leqs ~ \frac{c_4}\kappa  \sum_{k\geqs 1} \sum_{p\geqs k} \kappa^{p-k+1} \sum_{\tau \in \Sigma^k\cap \cA_p} \int_{f_\tau(Y^0)} {|\nabla v|}^2\dx x.
\end{equation}
Indeed, observe that the terms in the sum of \eqref{ineq-I2} for which $p=k$ are exactly $\kappa \int_{Y^k}{|\nabla v|}^2\dx x$.
\\[1mm]
Therefore, \eqref{ineq-I1} and \eqref{ineq-I2} yield, which \eqref{eq-lem2}, which achieves the proof.
\end{proof}

\begin{proof}[Proof of Theorem \ref{ext-thm}]
In the proof we will write $\lesssim$ when there may arise in the inequality a constant that does not depend on the function $v\in H^1(\Omegai)$ we consider.
\\
By Lemma \ref{lem2} and Theorem \ref{auscher-badr}, for every $v\in H^1(\Omegai)$ such that $v_{|\Gamma}\in H^{1/2}(\Gamma)$, there exists $\Lambda v\in H^1(\R^2)$ such that $(\Lambda v)_{|C}=v$ and
\begin{equation*}
{\Vert \Lambda v \Vert}_{H^1(\R^2)}^2 ~\lesssim ~ {\Vert v \Vert}_{H^1(C)}^2 + {\left\Vert \frac{v_a}\rho\right\Vert}_{L^2(C)}^2.
\end{equation*}
Define $\widehat {C} = \text{Int}(\overline C \cup \overline Y^0)$ (see Figure \ref{cone}). Introducing a cut-off function with support in the main hole $T$ and using the operator $\Lambda$ and Remark \ref{rem-AB}, we can construct a linear extension operator $\cF_0$ from $H^1(\widehat C)$ to $H^1(\widehat C \cup \overline T)$ such that $\cF 1 = 1$ and for all $v\in H^1(\widehat C)$
\begin{equation}\label{eq-th9}
\int_T {|\nabla (\cF_0 v)|}^2 \dx x  
~\lesssim ~\int_{\widehat C} {|\nabla v|}^2\dx x + \int_C {\bigg| \frac{v_a}{\rho}\bigg|}^2\dx x,
\end{equation}
which also implies
 \begin{equation}\label{a-b}
{\Vert \cF_0 v \Vert}_{H^1(\widehat C\cup \overline T)}^2 ~\lesssim ~ {\Vert v \Vert}^2_{H^1(\widehat C)} + {\left\Vert \frac{v_a}\rho\right\Vert}_{L^2(C)}^2.
\end{equation}
We will now define an extension $\tilde v\in H^1(\widetilde \Omega_\text{int})$ of a function $v$ as in Theorem \ref{ext-thm}, where $\widetilde \Omega_\text{int}$ is the convex hull of the domain $\Omegai$.
Recall that $T$ is the main hole and $\{f_\sigma(T)$, $\sigma\in \cA\}$ is the collection of the holes of the domain $\Omega_\text{int}$ (see Figure \ref{cone}).
Introduce the function $\tilde v$ defined in $\widetilde \Omega_\text{int}$ by
\begin{equation*}
\left\{ \begin{array}{ll}
	\tilde v := v &\text{in } \Omegai,\\
	\tilde v := \cF_0(v\circ f_\sigma)\circ {f_\sigma}^{-1} & \text{in } f_\sigma(T),~ \sigma \in \cA.
\end{array}\right.
\end{equation*} 
By Lemma \ref{lem2} and \eqref{eq-th9}, we get the estimate
\begin{equation}\label{estimation-T}
\int_T {|\nabla (\cF_0 v)|}^2 \dx x
~\lesssim ~
\sum_{k\geqs 1} 2^k \int_{\gamma_k} {|v-{\langle v \rangle}_{\gamma_k}|}^2\dx \mu
+ \sum_{k\geqs 0} \sum_{p\geqs k} \kappa^{p-k+1} \sum_{\tau \in \Sigma^k\cap \cA_p} \int_{f_\tau(Y^0)} {|\nabla v|}^2\dx x. 
\end{equation}
Indeed, $\kappa \int_{Y^0} {|\nabla v|}^2 \dx x$ (\textit{resp.} $\kappa \int_{C} {|\nabla v|}^2\dx x$) is bounded from above by the terms for which $k=p=0$ (\textit{resp.} $k\geqs 1$, $p=k$) in the second sum in \eqref{estimation-T}.
\\
Observe that for every integer $k\geqs 1$, $\gamma^k \subset \widetilde \Gamma^{\sigma^k}$ where $\sigma^k = (1,2,\ldots,2)\in \cA_{k-2}$, (recall that the sets $\widetilde \Gamma^{\sigma}$ have been introduced in Remark \ref{norme-Lip}).
Therefore,
\begin{equation}\label{ineq-lip}
\int_{\gamma_k} {|v-{\langle v \rangle}_{\gamma_k}|}^2 \dx \mu ~\lesssim~ \int_{\widetilde \Gamma^{\sigma^k}} {|v-{\langle v \rangle}_{\widetilde \Gamma^{\sigma^{k}}}|}^2\dx x,
\end{equation}
for $i=1,2$, where the constant in the inequality does not depend on $k$.
Take $\sigma\in \cA_n$, one has
\begin{eqnarray*}
\int_{f_\sigma(T)} {|\nabla \tilde v|}^2 \dx x 
&=& \int_{T} {|\nabla(\cF_0(v\circ f_\sigma))|}^2\dx x\\
&\lesssim & 
\sum_{k\geqs 1} 2^{k+n} \int_{f_\sigma(\widetilde \Gamma^{\sigma^k})} {|v-{\langle v \rangle}_{f_\sigma(\widetilde \Gamma^{\sigma^k})}|}^2\dx \mu \notag \\
&& ~~ + \sum_{k\geqs 0} \sum_{p\geqs k} \kappa^{p-k+1} \sum_{\tau \in \Sigma^k\cap \cA_p} \int_{f_{\sigma \tau}(Y^0)} {|\nabla v|}^2\dx x , \label{estimation-Tsigma}
\end{eqnarray*}
where we applied \eqref{estimation-T} to the function $v\circ f_\sigma$, and we used \eqref{ineq-lip}. The constant in the inequality does not depend on $n$. The notation $\sigma \tau$ for $\tau \in \cA_k$ stands for $(\sigma(1),\ldots,\sigma(n),\tau(1),\ldots,\tau(k))\in \cA_{n+k}$.\\
We can write 
\begin{eqnarray*}
\int_{\tildeOmegai} {|\nabla \tilde v|}^2 \dx x &=& \int_{\Omegai} {|\nabla \tilde v|}^2 \dx x + \sum_{\sigma\in \cA} \int_{f_\sigma(T)} {|\nabla \tilde v|}^2 \dx x\\
& \lesssim& \int_{\Omegai} {|\nabla \tilde v|}^2 \dx x + S_1+S_2,
\end{eqnarray*}
where
\begin{eqnarray*}
S_1 &=& \sum_{n\geqs 0} \sum_{\sigma\in \cA_n} \sum_{k\geqs 1} 2^{k+n} \int_{f_\sigma(\widetilde \Gamma^{\sigma^k})} {|v-{\langle v \rangle}_{f_\sigma(\widetilde \Gamma^{\sigma^k})}|}^2\dx \mu,\\
S_2 &=& \sum_{\sigma\in \cA} \sum_{k\geqs 0} \sum_{p\geqs k} \kappa^{p-k+1} \sum_{\tau \in \Sigma^k\cap \cA_p} \int_{f_{\sigma \tau}(Y^0)} {|\nabla v|}^2\dx x.
\end{eqnarray*}
We first deal with $S_1$. Take $\sigma,\tau\in \cA$ and $k,l\geqs 1$. Note that if $\sigma \sigma^k = \tau \sigma^l$, then $k=l$ and $\sigma=\tau$. Therefore,
\begin{equation}\label{S1}
S_1 ~\leqs~ \sum_{k\geqs 0} 2^k \sum_{\sigma\in \cA_k} \int_{\widetilde \Gamma^\sigma} {|v-{\langle v \rangle}_{\widetilde \Gamma^\sigma}|}^2\dx \mu ~\lesssim ~ {\Vert v_{|\Gamma} \Vert}_{H^{1/2}(\Gamma)}^2
\end{equation}
by \eqref{H-Lip}.
\\[1mm]
We are left with dealing with $S_2$. Assume that $\eta\in \cA_N$ and $\eta=\sigma \tau$ with $\tau\in \Sigma^k\cap \cA_p$, then $p\leqs N$. Since the sets $\Sigma^k$, $k\geqs 1$, are pairwise disjoint, this means that the term $\int_{f_\eta(Y^0)} {|\nabla v |}^2\dx x$ appears at most $N$ times in the sum $S_2$. 
\\
Moreover, we observe that $p-k+1 \in [1,N+1]$. It can be seen that there is at most one quadruplet $(\sigma',\tau',l,q)$ with $\sigma'\in \cA$, $\tau'\in \Sigma^l\cap \cA_{q}$, $l\geqs 1$ and $q\geqs l$ distinct from $(\sigma,\tau,k,p)$ such that $\eta=\sigma'\tau'$ and $p-k+1=q-l+1$. As a consequence,
\begin{eqnarray}
S_2 &\leqs& 2 \sum_{N\geqs 0} \sum_{\eta\in \cA_N} \sum_{m=1}^{N+1} \kappa^m \int_{f_\eta(Y^0)} {|\nabla v|}^2\dx x \notag\\
&\leqs & \frac 2{1-\kappa}  \sum_{\eta\in \cA} \int_{f_\eta(Y^0)} {|\nabla v|}^2\dx x
	~=~ \frac 2{1-\kappa} \int_{\Omegai} {|\nabla v|}^2\dx x. \label{S2}
\end{eqnarray}
Therefore, \eqref{S1} and \eqref{S2} give
$${\Vert \tilde v \Vert}_{H^1(\tildeOmegai)}^2 ~\lesssim~  {\Vert v \Vert}_{H^1(\Omegai)}^2 + {\Vert v_{|\Gamma} \Vert}_{H^{1/2}(\Gamma)}^2.$$
Since $\widetilde\Omega_\text{int}$ is a polygonal domain, we can further extend $\tilde v$ into a function $\cF v$ in $H^1(D)$, where $\cF$ is a linear operator satisfying \eqref{continuity}.
\end{proof}

\paragraph{Proof of Theorem~\ref{M-cv-theta=0}}
We will prove separately points \ref{point1} and \ref{point2} in Definition \ref{Mosco-convergence}.
\\[1mm]
\textit{Proof of point \ref{point1}}~ Suppose that ${(u_n)}$ weakly converges  to $u$ in $L^2(D)$. 
Without loss of generality, one can suppose $\limi a_n(u_n,u_n)$ is finite. We may further assume that there exists  a subsequence, still called $(u_n)$, such that 
 $a_n(u_n,u_n)$ converges to some real number as $n\to \infty$; as a consequence, 
there exists a constant  $c$ independent of $n$ such that
\begin{equation}\label{u-bound}
a_n(u_n,u_n) ~\leqs~ c.
\end{equation}
In particular, for all $n$, $u_n\in V^n$, which implies that $u_n\in V$.
Then, \eqref{u-bound} implies that $({u_n}_{|\Omegai})$ is bounded in $H^1(\Omegai)$ and that  $\sqrt {\nu_n}  \nabla {u_n}_{|\Omegae} $ is bounded in $L^2 (\Omegae)$.
Therefore, there exists a subsequence that we still denote $(u_n)$  such that 
\begin{itemize}
\item $({u_n}_{|\Omegai})$  converges to $u_{|\Omegai}$ weakly in $H^1(\Omegai)$, and strongly in $L^2(\Omegai)$
\item $\sqrt {\nu_n} \; \nabla {u_n}_{|\Omegae} $ converges  weakly in $L^2(\Omegae)$ to $\nabla {u}_{|\Omegae}$ (recall that $\nu_n$ converges to $1$ almost everywhere, so the weak limit of  $\sqrt {\nu_n} \; \nabla {u_n}_{|\Omegae} $  must be $\nabla {u}_{|\Omegae}$.)
\end{itemize}
Thus
\begin{displaymath}
  \begin{split}
& \limi \left(\int_{\Omegai^n}   \left( {|\nabla u_n|}^2  +\beta u_n^2 \right ) \dx x + \int_{\Omegae^n}   \left ( \nu_n {|\nabla u_n|}^2 + \beta u_n^2\right) \dx x \right) \\
\geqs ~  & \int_{\Omegai} \left({|\nabla u|}^2+\beta u^2\right) \dx x + \int_{\Omegae} \left({|\nabla u|}^2+\beta u ^2\right) \dx x.    
  \end{split}
\end{displaymath}
Moreover, exactly as in the proof of Theorem \ref{M-cv}, we see that 
 $\frac 1{|\Gamma^n|}\int_{\Gamma^n} u_n^2\dx x \to \int_\Gamma u^2\dx \mu$ as $n\to \infty$. We have proved point \ref{point1}
\\[2mm]
\textit{Proof of point \ref{point2}}~ Take $u\in L^2(D)$. By \eqref{def-a}, we may assume that $u\in V$.
 We must construct $(u_n)$ converging strongly in $L^2(D)$ such that \eqref{M-conv2} holds.
\\  Recall that $H$ is the height of the triangle $T$ and that $\Gamma^0$ is a segment of the line $\{x_2=0\}$, 
 $\Gamma$ is a segment of the line $\{x_2=2H\}$. We then introduce a sequence of smooth cut-off functions $\chi_n (x_2)$ such that 
$\chi_n(x_2)=0$ if $x_2\le  2H - \frac {3H} {2^{n+1}}$,
$\chi_n(x_2)=1$ if $x_2\in [2H- \frac H {2^{n}} , 2H ]$ and that $ 2^{-n} \|\chi_n' \|_{L^\infty}   + \|\chi_n \|_{L^\infty}$ is bounded
by a constant independent of  $n$.
\\[1mm]
We now define the  functions $u_n$ by
\begin{equation}
\label{u_n}
u_n(x)=\left\{
  \begin{array}[c]{ll}
\ui(x)\quad  & \forall x\in \Omegai, \\
\ue(x)\quad  & \forall x\in U, \\
\cE( \ui)(x) \quad  &\ds \forall x \in   \bigcup_{m\ge n} \bigcup_{\sigma\in \cA_m  } f_\sigma(T),\\
 \chi_n(x_2)  \cE( \ui)(x)   +  (1-\chi_n(x_2)) \ue (x)   \quad  &\ds \forall x\in  \bigcup_{m=0}^{n-1} \bigcup_{\sigma\in \cA_m  } f_\sigma(T),
  \end{array}
\right.
\end{equation}
where $\cE$ is the extension operator introduced in Theorem \ref{ext-thm}, $\ui = u_{|\Omegai}$ and $\ue = u_{|\Omegae}$.
It is easy to check that   $u_n$ belongs to the space $V^n$ and that the sequence $(u_n)$ strongly converges  to $u$ in $L^2(D)$.
We claim that 
\begin{eqnarray}
  \label{eq:9}
\sum_{m=0}^{n-1} \sum_{\sigma\in \cA_m  } \int_{f_\sigma(T)} \nu_n(x) \left| \nabla  \left ( \chi_n    ( \cE( \ui)  - \ue ) \right)\right|^2\dx x \longrightarrow 0, \\
\label{eq:10}
 \sum_{m\ge n} \sum_{\sigma\in \cA_m  } \int_{f_\sigma(T)} \left| \nabla  \left ( \cE( \ui)  - \ue  \right)\right|^2 \dx x\longrightarrow 0.
\end{eqnarray}
Indeed, we readily obtain (\ref{eq:10}) from the fact  that the measure of $ \bigcup_{m\ge n} \bigcup_{\sigma\in \cA_m  } f_\sigma(T)$ tends to zero and the fact that
 $\| \cE(\ui)\|_{H^1 (D)}$ and $\| \ue\|_{H^1 (D)}$ are finite. We obtain  (\ref{eq:9}) because $\|  \sqrt {\nu_n} \nabla \chi_n\|_{L^\infty} $ is bounded uniformly with respect to $n$ and 
$\chi_n$ is  supported in a region with vanishing measure, and because $\| \cE(\ui)\|_{H^1 (D)}$ and $\| \ue\|_{H^1 (\Omegae)}$ are finite.
\\
Therefore, 
\begin{displaymath}
  \begin{split}
&     \lim_{n\to \infty} \left( \int_{\Omegai^n} (   |\nabla \ui^n |^2  +\beta   |\ui^n |^2)  \dx x + 
\int_{\Omegae^n}  (\nu_n  |\nabla \ue^n|^2+ \beta |\ue^n |^2)\dx x \right)
\\
=&  \int_{\Omegai} ( |\nabla \ui|^2  +\beta   \ui ^2)  \dx x + \int_{\Omegae^n}  (  |\nabla \ue|^2 + \beta \ue^2  )\dx x,
  \end{split}
\end{displaymath}
where $\ui^n := {u_n}_{|\Omegai^n}$ and $\ue^n := {u_n}_{|\Omegae^n}$. Finally, from Proposition \ref{convergence-result}, $\frac{1}{|\Gamma^n|} \int_{\Gamma^n} u^2\dx x \longrightarrow \int_{\Gamma} u^2\dx \mu $
as $n\to \infty$. Collecting all the above results, we obtain that  $\lim a_n(u_n,u_n) = a(u,u)$ as $n\to \infty$, thus  point {\it ii}.

\nocite{MR3108826}

{\small
\bibliographystyle{plain}
\bibliography{biblio}}
\end{document}